\documentclass[12pt]{amsart}
\usepackage{amssymb}
\usepackage{graphicx}
\usepackage{fullpage}
\usepackage{amsthm}
\usepackage{amsmath}
\usepackage{setspace}
\usepackage{float}
\usepackage{placeins}
\usepackage{color}
\usepackage{leftidx}
\newtheorem{thm}{Theorem}[section]
\newtheorem{cor}[thm]{Corollary}
\newtheorem{lem}[thm]{Lemma}
\newtheorem{prop}[thm]{Proposition}
\theoremstyle{definition}
\newtheorem{defn}[thm]{Definition}
\newtheorem{example}[thm]{Example}
\theoremstyle{remark}
\newtheorem{rem}[thm]{Remark}
\numberwithin{equation}{section}

\begin{document}
\title{New Bernstein functions and subordination representations for time-nonlocal evolution equations}

\author{Calvin Khor}
\address{School of Mathematics, Southwest Jiaotong University, Chengdu 611756, P. R. China. }
\email{calvin\_khor@hotmail.com}

\author{Cheng-Gang Li}
\address{School of Mathematics, Southwest Jiaotong University, Chengdu 611756, P. R. China. }
\email{lichenggang@swjtu.edu.cn}

{\renewcommand{\thefootnote}{} \footnote{
\\ \text{  }  \ \    {\it Key words and phrases.} time-fractional differential equation, subordination principle, Bernstein function, inverse subordinator, Laplace transform.
 }}

\begin{abstract}
This paper establishes subordination representations from multi-term time-fractional differential equations to single-term ones, a direction complementary to the previously known subordination from single-term to multi-term equations; combining the two complementary directions also yields subordination representations between two multi-term equations. The crucial step is to prove that certain implicitly defined functions are Bernstein functions. We show that the quotient $\big(\sum_{j=1}^m a_j\lambda^{\alpha_j}\big)/\big(\sum_{j=1}^m a_j\alpha_j\lambda^{\alpha_j}\big)$ is a Stieltjes function whenever $0\leq\alpha_1-\alpha_m\leq 1$, and consequently that the inverse function of $\lambda=\big(\sum_{j=1}^m a_j\varphi^{\alpha_j}\big)^{1/\beta}$ is a Bernstein function for every $0<\beta\leq\alpha_m$. As applications, we obtain subordination representations with a clear probabilistic meaning: from multi-term to single-term time-fractional diffusion-wave equations, between abstract nonlocal-in-time equations with kernels of type $(\mathcal{PC})$ (including equations of tempered stable type), and between nonlocal-in-time telegraph equations. A necessary and sufficient condition for $[(\lambda+b)^\beta-b^\beta]^\gamma$ to be a Bernstein function also yields a negative result: in general, no such subordination exists from an $\alpha$-stable type dynamics to a tempered $\beta$-stable one when $\alpha<1$.
\end{abstract}
\maketitle


\section{Introduction and preliminaries}\label{sec:intro}

\subsection{Introduction and background}\label{subsec:background}

Fractional differential equations and their nonlocal-in-time generalizations have proved to be powerful tools for modeling memory effects and anomalous transport phenomena, and have found wide applications in physics, engineering, biology, finance and other fields; we refer to the classical monographs \cite{Diethelm2010book,Kilbas2006book,Meerschaert2012book,Podlubny1999book} and the eight-volume handbook series \cite{HFCA2019} for the general theory and applications.
This paper concentrates on subordination representations of time-fractional differential equations with a stochastic interpretation and their nonlocal-in-time generalizations; the crucial ingredients are some new completely monotone functions and Bernstein functions.
Since all of our equations can be transformed into the abstract form of the Volterra equation \cite{Pruss1993book}
\begin{equation}\label{eq:Volterra}
u(t) = x + \int_0^t a (t-s) Au(s) ds, \qquad x\in X, \quad a(\cdot)\in L^1_{loc}(\mathbb{R}_+),
\end{equation}
where $X$ is an appropriate Banach space and $A$ is a closed densely defined linear operator in $X$,
we use \eqref{eq:Volterra} to illustrate our main motivation in a formal way.

Denote the Laplace transform of $a(t)$ by $\widehat{a}(\lambda)=\int_0^\infty e^{-\lambda t}a(t)dt$. Suppose
\[
u(t) = x + \int_0^t a (t-s) Au(s) ds, \qquad v(t) = x + \int_0^t b (t-s) Av(s) ds
\]
are well-posed. If there exists a Bernstein function $\varphi(\lambda)$
such that $\varphi(0+)=0$ and $\widehat{a}(\varphi(\lambda))=\widehat{b}(\lambda)$, then we have the subordination representation
\begin{equation}\label{eq:subordination-general}
v(t)=\int_0^\infty u(s)P_\varphi (ds,t),
\end{equation}
where $P_\varphi (s,t)$ is the distribution function of an inverse subordinator $\mathbf{E}_\varphi(t)$ with Laplace exponent $\varphi$;
see Proposition \ref{prop:Pruss-subordination} below and \cite{Lichenggang2026dual,Pruss1993book} for details. If $P_\varphi (s,t)$ possesses a Lebesgue density $p_\varphi(s,t)$ for each $t>0$, then
$p_\varphi(s,t)$ satisfies $\int_0^\infty e^{-\lambda t}p_\varphi(s,t)dt=\frac{\varphi(\lambda)}{\lambda}e^{-s\varphi(\lambda)}$.

Clearly, a subordination relation \eqref{eq:subordination-general} cannot exist for every pair of kernels $a(\cdot)$ and $b(\cdot)$:
{\it to obtain a subordination representation \eqref{eq:subordination-general}, this construction requires a Bernstein function
$\varphi(\lambda)$ such that $\widehat{a}(\varphi(\lambda))=\widehat{b}(\lambda)$.} Therefore, the kernels $a(\cdot)$ and $b(\cdot)$ have to be specifically chosen.

Let $m$ be a positive integer, $0<\beta\leq 2$, $0<\alpha_m<\cdots<\alpha_2<\alpha_1\leq2$ and $a_j>0$ ($j=1,2,\cdots, m$).
Consider the abstract single-term time-fractional differential equation (TFDE) \cite{Bajlekova2001}
\begin{equation}\label{eq:STFDE}
\begin{aligned}
\begin{cases}
    D_t^\beta v(t)=Av(t),\quad\quad t>0; \\
    v(0) = \psi, \quad v'(0)=0\ (\text{if }\beta>1)
\end{cases}
\end{aligned}
\end{equation}
and the multi-term TFDE \cite{Lichenggang2012}
\begin{equation}\label{eq:MTFDE}
\begin{aligned}
\begin{cases}
    \sum\limits_{j=1}^m a_j D_t^{\alpha_j}u(t)=Au(t), \quad\quad t>0;\\
     u(0)=\psi, \quad u'(0)=0 \ (\text{if } \alpha_1>1),
\end{cases}
\end{aligned}
\end{equation}
where the Caputo fractional derivative $D_t^\alpha $ is defined in Definition \ref{def:Caputo}.

The well-posedness of \eqref{eq:MTFDE} can be derived from the well-posedness of \eqref{eq:STFDE} if $\beta\geq \alpha_1$ \cite{Lichenggang2012}.
Moreover, in some special cases, the solution of \eqref{eq:MTFDE} can be represented as
\begin{equation*}
u(t)=\int_0^\infty v(s)P_\varphi(ds,t), \quad t>0,
\end{equation*}
where $P_\varphi(ds,t)$ is the distribution of an inverse subordinator $\mathbf{E}_\varphi(t)$ with Laplace exponent $\varphi(\cdot)$, i.e.,
$$
\int_0^\infty e^{-\lambda t}P_\varphi(ds,t)dt=\frac{\varphi(\lambda)}{\lambda} e^{-s \varphi(\lambda)}ds;
$$
whenever $\mathbf{E}_\varphi(t)$ admits a Lebesgue density $p_\varphi(s,t)$ (see Remark \ref{rem:density-allt}), the formula reads
$u(t)=\int_0^\infty p_\varphi(s,t)v(s)ds$.
We cite some typical results in this direction (subordination from \eqref{eq:STFDE} to \eqref{eq:MTFDE}):

(1) If $m=1, a_1=1$, $\alpha_1<\beta$, then $\varphi(\lambda)=\lambda^{\alpha_1/\beta}$
\cite{Bajlekova2001,Limiao2010}.
In particular, if $\beta=2$, $\alpha_1=1$, it reduces to the classical Weierstrass formula \cite{Arendt2010book,Romanov1947}.

(2) If $m=2$, $\beta=\alpha_1=2$, $\alpha_2=1$, $a_1=1$, then \eqref{eq:MTFDE} reduces to the telegraph equation and
$\varphi(\lambda)=\sqrt{\lambda^2+a_2\lambda}$ \cite{Lichenggang2019,Orsingher2004}. See also the classical variant in \cite{Kac1974}.

(3) If $\beta=1$, $0<\alpha_m<\cdots<\alpha_2<\alpha_1\leq 1$, then $\varphi(\lambda)=\sum\limits_{j=1}^m a_j \lambda^{\alpha_j}$ \cite{Gorenflo2013,Hahn2012,Orsingher2017}.

(4) If $\beta\geq\alpha_1$, $0< \alpha_m<\cdots<\alpha_2<\alpha_1\leq 2$, $\alpha_1-\alpha_m\leq 1 $, then, based on the idea in \cite{Bajlekova2018}, one can take
$\varphi(\lambda)=\left(\sum\limits_{j=1}^m a_j \lambda^{\alpha_j}\right)^{1/\beta}$.

All of the above results concern subordination from a single-term equation to another single-term equation or to a multi-term equation.
For a class of multi-term fractional Jeffreys-type equations with special structure, subordination results with respect to certain evolution equations of integer order were established in \cite{Bazhlekova2024FCAA}, again in the direction from a simpler equation to a more complicated one.
However, to the best of our knowledge, there has been no subordination representation from the general multi-term equation \eqref{eq:MTFDE} to the single-term equation \eqref{eq:STFDE} so far.
Continuing our previous work on the subordination from the telegraph equation to the diffusion equation \cite{Lichenggang2026dual},
we establish the subordination formula from \eqref{eq:MTFDE} to \eqref{eq:STFDE} in several cases,
under the crucial condition $\beta\leq \alpha_m$, which is complementary to the condition $\beta\geq\alpha_1$ for the opposite direction.
Combining the two complementary directions further yields a subordination representation between two
multi-term TFDEs (Theorem \ref{thm:multi-to-multi}).

Furthermore, for $0<\beta<1$, the time-fractional differential equation \eqref{eq:STFDE} can be generalized to the nonlocal-in-time equation
\begin{equation*}
\begin{cases}
D_t(k\ast (u(\cdot)-u_0))(t)=Au(t), \qquad t>0; \\
u(0)=u_0,
\end{cases}
\end{equation*}
where the kernel $k(t)$ is a function of type $(\mathcal{PC})$ (see Definition \ref{def:PC-kernel}), the choice $k=g_{1-\beta}$ recovering \eqref{eq:STFDE}. Motivated by the subordination principle, we also construct representations of the form
\[
v(t)=\int_0^\infty u(s)P_\varphi (ds,t)
\]
between two such equations with kernels $k(t)$ and $l(t)$, including equations of tempered stable type, and between their telegraph-type counterparts (Subsection \ref{subsec:telegraph}), through which equations with time orders between $1$ and $2$ are covered as well.

The main contributions of this paper are as follows:
\begin{itemize}
\item[(i)] Several new criteria for Bernstein functions are provided, including a new Stieltjes function (Theorem \ref{thm:Stieltjes-quotient}), the Bernstein property of inverse functions of $\lambda=\big(\sum_{j=1}^m a_j\varphi^{\alpha_j}\big)^{1/\beta}$ (Theorem \ref{thm:inverse-multiterm} and Corollary \ref{cor:relax-beta}), the complete Bernstein property of $(\lambda^\beta+h)^\gamma$ (Proposition \ref{thm:CBF-tempered-power}), a necessary and sufficient condition for $[(\lambda+b)^\beta-b^\beta]^\gamma$ to be a Bernstein function (Proposition \ref{prop:tempered-iff}), a necessary and sufficient condition for $\big[(\lambda+b)^\beta+a^\alpha-b^\beta\big]^{1/\alpha}-a$ to be a Bernstein function (Proposition \ref{prop:shifted-power-iff}), which in particular settles the remaining case of the subordination between two tempered stable type equations (see Example \ref{ex:tempered-to-tempered} and Remark \ref{rem:open-tempered}), and a shared-outer-symbol criterion for the Bernstein property of $F^{-1}\circ G$ (Proposition \ref{prop:shared-outer});
\item[(ii)] the subordination representation from \eqref{eq:MTFDE} to \eqref{eq:STFDE} is established for all $0<\beta\leq\alpha_m$ (Theorem \ref{thm:MTFDE-to-STFDE}), and is combined with the known reverse direction to obtain a subordination representation between two multi-term TFDEs (Theorem \ref{thm:multi-to-multi});
\item[(iii)] subordination representations between general nonlocal-in-time equations with $(\mathcal{PC})$ kernels and between nonlocal-in-time telegraph equations are obtained, together with a negative result for the tempered stable case (Section \ref{sec:general}).
\end{itemize}

The rest of this paper is organized as follows. In Subsection \ref{subsec:prelim}, we fix the notation and
the solution concept, recall the well-posedness theory of the abstract time-fractional Cauchy problem
\eqref{eq:STFDE}, and collect the required facts on completely monotone functions, Bernstein functions,
Stieltjes functions, subordinators, and the subordination principle. Section \ref{sec:BF} is devoted to the new Bernstein functions. The subordination from multi-term to single-term TFDEs, as well as the resulting subordination between two multi-term TFDEs, is established in Section \ref{sec:multiterm}. Section \ref{sec:general} treats the general nonlocal-in-time equations and the nonlocal-in-time telegraph equations.

\subsection{Preliminaries and notation}\label{subsec:prelim}

Throughout this paper, $X$ is a Banach space, $A$ is a densely defined closed linear operator on $X$,
and $B(X)$ denotes the set of all bounded linear operators on $X$.
$L^1(\mathbb{R}_+; X)$ denotes the space of all Bochner-measurable functions
$u: \mathbb{R}_+ \rightarrow X$ such that $\|u(\cdot)\|$ is integrable.

Methods of the vector-valued Laplace transform \cite{Arendt2010book,Pruss1993book} are frequently used in this paper.
Suppose an $X$-valued function $f(t)$ is exponentially bounded, i.e., there exist constants $M>0$, $\omega\in \mathbb{R}$
such that $\|f(t)\|\leqslant Me^{\omega t}$. In particular, if $\omega=0$, then $f(t)$ is called bounded.
The Laplace transform of $f$ is denoted by
$$
\widehat{f}(\lambda)=\int_0^\infty e^{-\lambda t}f(t)\,dt, \qquad \operatorname{Re}\lambda>\omega.
$$
For $f\in L^1(\mathbb{R}_+; \mathbb{R})$ and $g\in L^1(\mathbb{R}_+; X)$, the Laplace convolution is defined by
$$
(f\ast g)(t)=\int_0^t f(t-s)g(s)ds=\int_0^t f(s)g(t-s)ds, \quad t>0.
$$
Clearly,
$$
\widehat{(f\ast g)} (\lambda)=\widehat{f}(\lambda)\widehat{g}(\lambda).
$$

\begin{defn}\label{def:Caputo}
The Caputo fractional derivative is defined by
\begin{eqnarray*}
(D_t^\alpha f)(t)=
\begin{cases}
\frac{1}{\Gamma(n-\alpha)}\int_0^t (t-s)^{n-\alpha-1} f^{(n)}(s)ds,  \quad\quad \text{for } n-1<\alpha<n;\\
f^{(n)}(t), \quad\quad\quad\quad \text{for } \alpha=n,
\end{cases}
\end{eqnarray*}
where $n=\lceil\alpha\rceil$ is the smallest integer greater than or equal to $\alpha>0$.
\end{defn}

The following formula for the Laplace transform of the Caputo fractional derivative will be used frequently in this paper:
\begin{equation}\label{eq:Caputo-Laplace}
\widehat{D_t^\alpha f}(\lambda)=\lambda^\alpha \widehat{f}(\lambda)-\sum_{j=0}^{\lceil\alpha\rceil-1}f^{(j)}(0)\lambda^{\alpha-1-j}.
\end{equation}

We next recall the well-posedness theory of the abstract Cauchy problem with a Caputo time-fractional
derivative, i.e., equation \eqref{eq:STFDE} with $0<\beta\leq2$, which is the standing single-term
equation of this paper; we refer to \cite{Bajlekova2001} for a systematic treatment. The
well-posedness of \eqref{eq:STFDE} is characterized by the existence of a $\beta$-times resolvent
family (solution operator) $\{S_\beta(t)\}_{t\geq0}\subset B(X)$, in which case the unique solution
is given by $v(t)=S_\beta(t)\psi$, and
\begin{equation}\label{eq:resolvent-family-Laplace}
\int_0^\infty e^{-\lambda t}S_\beta(t)dt=\lambda^{\beta-1}(\lambda^\beta-A)^{-1},
\qquad \lambda^\beta\in\rho(A).
\end{equation}
For $0<\beta\leq1$, every generator $A$ of a bounded
$C_0$-semigroup $T(t)$ also generates such a $\beta$-times resolvent family, given by the subordination formula
\begin{equation}\label{eq:S-beta-subordination}
S_\beta(t)=\int_0^\infty t^{-\beta}h_\beta(st^{-\beta})T(s)ds,
\end{equation}
where $h_\beta$ is the one-sided stable density with Laplace transform
$\int_0^\infty e^{-\lambda s}h_\beta(s)ds=e^{-\lambda^\beta}$. For $1<\beta<2$, the analogous statement
holds with bounded cosine families in place of $C_0$-semigroups \cite{Bajlekova2001}.

Next, we recall the theory of completely monotone functions and Bernstein functions \cite{Applebaum2009,Pruss1993book,Schilling2012}.

\begin{defn}\label{def:CMF-BF}
  A function $f(\cdot): (0,\infty)\rightarrow \mathbb{R}$ of class $C^\infty$ is said to be a
  {\it completely monotone function} if
  $$
    (-1)^n f^{(n)}(\lambda)\geq 0
  $$
  for all $n=0,1,2,\cdots$ and $\lambda >0$. Moreover, $\varphi(\cdot)$ is said to be a {\it Bernstein function}
  if $\varphi(\cdot)\geq0$ and $\varphi'(\cdot)$ is a completely monotone function.
\end{defn}

We write $\mathcal{CMF}$ and $\mathcal{BF}$ for the classes of completely monotone functions and Bernstein functions, respectively.
Recall that $e^{-a\lambda^\alpha}$ ($a>0$, $0\leq \alpha\leq 1$) and $\lambda^{-\alpha}$ ($\alpha>0$) are completely monotone, whereas
$\lambda^{\alpha}$  ($0<\alpha\leq 1$), $1-e^{-\lambda}$ and $\ln(1+\lambda)$ are Bernstein functions.
Moreover, $\mathcal{CMF}$ and $\mathcal{BF}$
appear in various fields of mathematics, such as integro-differential equations, probability, harmonic analysis, complex analysis, matrix analysis,
operator theory, etc. For example, the complete monotonicity of the Mittag-Leffler function $E_{\alpha}(-x)$ ($0<\alpha\leq1$), proved by
Pollard \cite{Pollard1948BAMS}, plays an important role in the positivity of time-fractional diffusion equations.

We recall the classical representation theorems for these two classes of functions.

\begin{prop} \label{prop:LK-representation}
  (1) Let $f$ be $\mathcal{CMF}$.
  Then it is the Laplace transform of a unique (nonnegative) measure $\nu$ on $[0,+\infty)$, i.e., for all $\lambda >0$,
  \[
    f(\lambda)=\mathcal{L} (\nu;\lambda)=\int_{[0,\infty)}e^{-\lambda t}\nu(dt).
  \]
  Conversely, whenever $ \mathcal{L}(\nu;\lambda)<\infty$ for every $\lambda >0$, $ \lambda \rightarrow\mathcal{L}(\nu;\lambda)$ is
  $\mathcal{CMF}$.

  (2) A function $\varphi(\cdot): (0,+\infty)\rightarrow (0,+\infty)$ is $\mathcal{BF}$ if and only if it
  admits the representation
  \[
    \varphi(\lambda)=a+b\lambda+\int_{(0,\infty)}(1-e^{-\lambda t})\sigma(dt),
  \]
  where $a,b\geq 0$ and $\sigma$ is a measure on $(0,\infty)$ satisfying
  $\int_{(0,\infty)}(1\wedge t)\sigma(dt)<\infty$. Furthermore, we have
  \[
    a=\lim_{\lambda\to0+} \varphi(\lambda), \qquad b=\lim_{\lambda\to\infty} \frac{\varphi(\lambda)}{\lambda}.
  \]
  We call $\sigma$ the L{\'e}vy measure.
\end{prop}

The definitions and characterization of $\mathcal{CMF}$ and $\mathcal{BF}$ are elementary; however, their verification is often non-trivial; see, for example, the proof of the BMV conjecture \cite{Stahl2013ACTA}.

The following closure properties of $\mathcal{CMF}$ and $\mathcal{BF}$ will be used freely throughout the paper \cite{Schilling2012}.
\begin{prop} \label{prop:closure}
  \begin{itemize}
    \item[(i)]
          If $f$ and $g$ are $\mathcal{CMF}$, then $f+g$ and $fg$ are $\mathcal{CMF}$.
    \item[(ii)]
          If $f$ and $g$ are $\mathcal{BF}$, then $f+g$ and $f\circ g$ are $\mathcal{BF}$.
    \item[(iii)]
          If $f$ is $\mathcal{CMF}$ and $g$ is $\mathcal{BF}$, then $f\circ g$ is $\mathcal{CMF}$.
    \item[(iv)]
          If $f$ is $\mathcal{BF}$, then $f(\lambda)/\lambda$ is $\mathcal{CMF}$.
  \end{itemize}
\end{prop}

The following subclass of $\mathcal{BF}$ will be useful \cite{Schilling2012}.

\begin{defn}\label{def:CBF}
A $\mathcal{BF}$ $f$ is said to be a {\it complete Bernstein function} if
its L{\'e}vy measure $\sigma$ has a completely monotone density $m(t)$ with respect to
Lebesgue measure:
\[
f(\lambda)=a+b\lambda+\int_{(0,\infty)}(1-e^{-\lambda t})m(t)dt.
\]
\end{defn}

We write $\mathcal{CBF}$ for the class of complete Bernstein functions.
For example, $\lambda^{\alpha}$ ($0\leq \alpha\leq 1$) and $\ln(1+\lambda)$ are $\mathcal{CBF}$, while
 $1-e^{-\lambda}$ is not $\mathcal{CBF}$.

Complete Bernstein functions have useful holomorphic extensions; see \cite[Theorem 6.2]{Schilling2012}.

\begin{prop}\label{prop:CBF-characterization}
Suppose that $f$ is a nonnegative function on $(0,\infty)$. Then the following
conditions are equivalent:
\begin{itemize}
\item[(i)]
 $f$ is $\mathcal{CBF}$.
\item[(ii)]
 $f$ has a holomorphic extension to $H^+=\{z\in \mathbb{C}; \operatorname{Im}\,z>0 \}$ such that $\operatorname{Im}\,f(z)\geq 0$ for all $z\in H^+$ and
such that the limit $f(0+)=\lim_{(0,\infty)\ni \lambda\rightarrow0} f(\lambda)$ exists and is real.
\item[(iii)]
$f$ has a holomorphic extension to the cut complex plane $\mathbb{C}\setminus(-\infty,0]$, and
$\operatorname{Im}\,z\cdot \operatorname{Im}\,f(z)\geq 0$ for every $z$ there; moreover,
the limit $f(0+)=\lim_{(0,\infty)\ni \lambda\rightarrow0} f(\lambda)$ exists and is real.
\end{itemize}
\end{prop}

Throughout this paper, we do not distinguish between $f$ and its analytic continuation.

\begin{defn}\label{def:Stieltjes}
A {\it Stieltjes function} is a function $f:(0,\infty)\to [0,\infty)$ which can be written in the form
\[
f(\lambda)=\frac{a}{\lambda}+b+\int_{(0,\infty)}\frac{1}{\lambda+t} \sigma(dt),
\]
where $a,b$ are nonnegative constants and $\sigma$ is a measure on $(0,\infty)$ such that
$\int_{(0,\infty)} (1+t)^{-1}\sigma(dt)<\infty$.
\end{defn}

Every Stieltjes function is completely monotone. We will need the following two facts from \cite{Schilling2012}.

\begin{lem}[{\cite[Corollary 7.4]{Schilling2012}}]\label{lem:Stieltjes-criterion}
Let $g$ be a positive function on $(0,\infty)$. Then $g$ is a Stieltjes function if, and only if, $g(0+)$ exists in $[0,\infty]$ and $g$ extends analytically
to $\mathbb{C}\setminus (-\infty,0]$ such that $\operatorname{Im}\,z\cdot \operatorname{Im}\,g(z)\leq 0$.
\end{lem}

\begin{rem}\label{rem:CBF-reciprocal}
A (nonzero) function $f$ is $\mathcal{CBF}$ if, and only if, $1/f$ is a (non-trivial) Stieltjes function \cite[Theorem 7.3]{Schilling2012}. In particular, the reciprocal of a complete Bernstein function is completely monotone. This fact will be used repeatedly in Section \ref{sec:BF}.
\end{rem}

We next recall the probabilistic preliminaries needed below.

\begin{defn}\label{def:subordinator}
A {\it subordinator} $\mathbf{D}(t)$
  is a one-dimensional L{\'e}vy process taking values in $[0,\infty)$ that is nondecreasing and starts at the origin almost surely.
\end{defn}

As a special L{\'e}vy process, a subordinator admits a simple L{\'e}vy--Khintchine representation.

\begin{thm}\label{thm:LK-subordinator}
  If $\mathbf{D}(t)$ is a subordinator, then the Laplace transform of its distribution is
  \[
    \mathbb{E}(e^{-\lambda \mathbf{D}(t)})=e^{-t\varphi(\lambda)},
  \]
  and the Laplace exponent $\varphi$ takes the form
  \begin{equation}\label{eq:LK-subordinator}
    \varphi(\lambda)=b\lambda+\int_{(0,\infty)}(1-e^{-\lambda s})\sigma(ds),
  \end{equation}
  where $b\geq 0$ and the measure $\sigma$ satisfies the same condition as in Proposition \ref{prop:LK-representation}.

  Conversely, any function $\varphi:(0,\infty)\to[0,\infty)$ of the form \eqref{eq:LK-subordinator} is the Laplace exponent of a subordinator.
\end{thm}

\begin{rem}\label{rem:BF-Laplace-exponent}
  There is a one-to-one correspondence between Bernstein functions $\varphi$ with $\varphi(0+)=0$ and
  subordinators \cite{Applebaum2009}: the Laplace exponent of every subordinator is such a Bernstein function,
  and conversely, every such Bernstein function is the Laplace exponent of a unique (in law) subordinator
  $\mathbf{D}_{\varphi}(t)$. Hereafter we use the term {\it Laplace exponent} in probabilistic contexts
  and {\it Bernstein function} in analytic contexts; both terms refer to the same object.
\end{rem}

Throughout this paper, we denote by $\mathbf{D}_{\varphi}(t)$ the subordinator with Laplace exponent $\varphi(\cdot)$. Its first passage time is defined by
\begin{equation}\label{eq:inverse-subordinator}
  \mathbf{E}_{\varphi}(t)=\inf\{z>0: \mathbf{D}_{\varphi}(z)>t\}.
\end{equation}
We call $\mathbf{E}_{\varphi}(t)$ the {\it inverse subordinator} with Laplace exponent $\varphi(\cdot)$. By the monotonicity and right-continuity of $\mathbf{D}_{\varphi}$, the exact relation
$\{\mathbf{E}_{\varphi}(t)\geq s\}=\{\mathbf{D}_{\varphi}(s-)\leq t\}$
holds for all $s,t>0$; in particular,
\[
\mathbb{P}(\mathbf{E}_{\varphi}(t)> s)=\mathbb{P}(\mathbf{D}_{\varphi}(s)\leq t)
\]
whenever $\mathbb{P}(\mathbf{D}_{\varphi}(s)=t)=0$. The latter holds for all $s,t>0$ when the L{\'e}vy measure of $\varphi$ is infinite, as in the cases covered by Remark \ref{rem:density-allt} below.
When it exists, we denote by $p_{\varphi}(s,t)$ the (Lebesgue) probability density of
$\mathbf{E}_{\varphi}(t)$.

\begin{rem}\label{rem:density-allt}
Suppose that $\varphi$ is a driftless Laplace exponent with $\varphi(\infty)=\infty$; equivalently, its
L{\'e}vy measure $\sigma$ is infinite: $\sigma(0,\infty)=\infty$. Then, for {\rm every} $t>0$,
$\mathbf{E}_{\varphi}(t)$ admits a Lebesgue density $p_{\varphi}(s,t)$, given by
\[
p_{\varphi}(s,t)=\int_0^t \sigma(t-\tau,\infty)\,\mathbb{P}(\mathbf{D}_{\varphi}(s)\in d\tau),
\]
see \cite[Theorem 3.1]{Meerschaert2008} (for a related almost-everywhere statement, see also
\cite[Lemma 2.1]{ChenZhenQing2020JFA}). All Laplace exponents arising as the final time change in the
subordination results of this paper are driftless with $\varphi(\infty)=\infty$ (this is verified at each
occurrence), so the corresponding inverse subordinators admit Lebesgue densities for every $t>0$, and all
subordination formulas below hold in the density form for every $t>0$. The only exception is the
intermediate exponent $\varphi_2$ in Theorem \ref{thm:multi-to-multi}, whose boundary case
$\beta=\gamma_1$ has a positive drift; Remark \ref{rem:p2-density} shows that the composed exponent
$F^{-1}\circ G$ is nevertheless driftless with an infinite L{\'e}vy measure, so the final kernel
$p_0(s,t)$ there is a genuine probability density for every $t>0$.
\end{rem}

We will also need the following standard Laplace-transform identity for inverse subordinators.

\begin{prop}\label{prop:inverse-subordinator-Laplace}
Let $\mathbf{E}_{\varphi}(t)$ be the inverse subordinator with Laplace exponent $\varphi$.
Then for any $\lambda>0$ and any bounded measurable function $F:[0,\infty)\to\mathbb{R}$,
\[
\int_0^\infty e^{-\lambda t}\, \mathbb{E}\bigl[F(\mathbf{E}_{\varphi}(t))\bigr]\,dt
= \frac{\varphi(\lambda)}{\lambda} \int_0^\infty F(s)\, e^{-s\varphi(\lambda)}\,ds.
\]
In particular, if $\mathbf{E}_{\varphi}(t)$ has a density $p_\varphi(\cdot,t)$,
then the Laplace transform of $p_{\varphi}(s,t)$ in $t$ is
\[
  \int_0^\infty e^{-\lambda t} p_{\varphi}(s,t)dt=\dfrac{{\varphi}(\lambda)}{\lambda}e^{-s{\varphi}(\lambda)}.
\]
\end{prop}

Finally, we recall the class of kernels of type $(\mathcal{PC})$ and the subordination principle.

\begin{defn}[{\cite{Kemppainen2016}}]\label{def:PC-kernel}
Let $k\in L^1_{loc}(\mathbb{R}_+)$ be nonnegative and nonincreasing. If there exists a function $k_1\in L^1_{loc}(\mathbb{R}_+)$ such that
$(k\ast k_1)(t)=1$ for all $t>0$, then $k$ is said to be a function of type $(\mathcal{PC})$. We also write $(k,k_1)\in (\mathcal{PC})$.
\end{defn}

The classical example is the pair $(g_{1-\alpha}, g_\alpha)$ ($\alpha\in (0,1)$), where
\[
g_{\alpha}(t)=\frac{t^{\alpha-1}}{\Gamma(\alpha)}, \qquad t>0.
\]
Note that $\widehat{g_\alpha}(\lambda)=\lambda^{-\alpha}$, and if $(k,k_1)\in(\mathcal{PC})$, then
$\widehat{k_1}(\lambda)=\frac{1}{\lambda\widehat{k}(\lambda)}$. Kernels of type $(\mathcal{PC})$ and the associated
nonlocal-in-time equations have been systematically studied; see \cite{Kemppainen2016,Kochubei2011,LuchkoYamamoto2016,Pruss1993book}.
We remark that in the fractional calculus community, a pair $(k,k_1)$ satisfying $(k\ast k_1)(t)=1$ is usually called a \emph{Sonine pair}, and the identity $k\ast k_1=1$ is referred to as the \emph{Sonine condition} (see, e.g., \cite{Kochubei2011,LuchkoYamamoto2016}); thus kernels of type $(\mathcal{PC})$ are precisely the Sonine kernels with a nonnegative nonincreasing first component.

\begin{rem}\label{rem:PC-boundary}
The Dirac delta $\delta_0$ is not of type $(\mathcal{PC})$ in the sense of Definition
\ref{def:PC-kernel}, since it is not locally integrable (although the pair $(\delta_0,1)$ still
satisfies the Sonine condition in the measure sense). In the examples of Section
\ref{sec:general} the boundary value $\alpha=1$ corresponds to this degenerate case: the equation
$D_t(k\ast(u(\cdot)-u_0))=Au$ with $k=\delta_0$ reduces to the classical equation $u'=Au$, and
its telegraph-type counterpart reduces to the classical telegraph equation; the corresponding
subordination identities remain valid, as one can check directly by the Laplace transform or by
approximation with $\alpha\uparrow1$. We include this boundary case in the examples without
further mention.
\end{rem}

We conclude with the solution concept used in the sequel. By a {\it strong solution} of the equation
$D_t(k\ast(u(\cdot)-u_0))(t)=Au(t)$, $u(0)=u_0$, with a kernel $k$ of type $(\mathcal{PC})$, we mean a
function $u\in C([0,\infty);X)$ with $u(0)=u_0$ such that $k\ast(u(\cdot)-u_0)$ is locally absolutely
continuous, $u(t)\in D(A)$ and the equation holds for almost every $t>0$. Strong solutions of the
telegraph-type equations in Subsection \ref{subsec:telegraph} are understood analogously, with
$k\ast k\ast(u(\cdot)-u_0)$ being continuously differentiable, vanishing together with its derivative at
$t=0$, and twice differentiable almost everywhere. For orders in $(0,1)$, the equations \eqref{eq:STFDE} and
\eqref{eq:MTFDE} fit into this framework with $k=g_{1-\alpha}$; for higher orders the Caputo sense of
Definition \ref{def:Caputo} is used. All Laplace transform computations in this paper are justified for
exponentially bounded strong solutions; e.g., $D_t(k\ast(u(\cdot)-u_0))=Au$ transforms into
$\lambda\widehat{k}(\lambda)\big(\widehat{u}(\lambda)-u_0/\lambda\big)=A\widehat{u}(\lambda)$. For existence,
uniqueness and regularity results, see \cite{Kemppainen2016,Kochubei2011,LuchkoYamamoto2016,Pruss1993book}.

\begin{prop}[{\cite[Theorem 4.1]{Pruss1993book}}]\label{prop:Pruss-subordination}
Let $a(\cdot), b(\cdot)\in L^1_{loc}(\mathbb{R}_+)$ and suppose the Volterra equations
\[
u(t) = x + \int_0^t a (t-s) Au(s) ds, \qquad v(t) = x + \int_0^t b (t-s) Av(s) ds
\]
admit bounded strong solutions $u$ and $v$, respectively. If there exists a Laplace exponent $\varphi(\lambda)$
such that $\widehat{a}(\varphi(\lambda))=\widehat{b}(\lambda)$ for all $\lambda>0$, then
\[
  v(t)=\int_0^\infty u(s)P_\varphi(ds,t),
\]
where $P_\varphi(ds,t)$ is the distribution of the inverse subordinator $\mathbf{E}_\varphi(t)$, satisfying
\[
\int_0^\infty e^{-\lambda t}P_\varphi(ds,t)dt=\frac{\varphi(\lambda)}{\lambda}e^{-s\varphi(\lambda)}ds.
\]
In particular, if $P_\varphi(s,t)$ has a Lebesgue density $p_\varphi(s,t)$, then
\[
  v(t)=\int_0^\infty p_\varphi(s,t)u(s)\,ds.
\]
\end{prop}

Note that Proposition \ref{prop:Pruss-subordination} is a variant of \cite[Theorem 4.1]{Pruss1993book}:
the latter is formulated in the language of resolvent families for Volterra equations, whereas the
present formulation, phrased directly in terms of the distribution of the inverse subordinator
$\mathbf{E}_\varphi(t)$, follows that of \cite{Lichenggang2026dual}.

\begin{rem}\label{rem:verification-form}
Proposition \ref{prop:Pruss-subordination} is stated in the verification form: both solutions are assumed
to exist, and the subordination identity is proved. Under standard generation assumptions on $A$, the
formula $v(t)=\int_0^\infty u(s)P_\varphi(ds,t)$ can also be used to construct the solution of the
$b$-equation from that of the $a$-equation; this is the point of view of \cite[Theorem
4.1]{Pruss1993book}, where resolvent families are built by subordination. We adopt the verification form
throughout in order to focus on the main difficulty---the Bernstein property of the implicitly defined
Laplace exponents; existence results are recalled in Remarks \ref{rem:STFDE-existence} and
\ref{rem:multi-multi-existence}.
\end{rem}

Although the equations in Sections \ref{sec:multiterm}--\ref{sec:general} appear in the form of time-differential equations,
the underlying idea of all subordination constructions in this paper stems from Proposition \ref{prop:Pruss-subordination}.
When applying this principle, the main difficulty lies in determining whether the (usually implicitly defined) candidate
function $\varphi(\lambda)$ is a Bernstein function; this is the topic of the next section.

\section{New Bernstein functions}\label{sec:BF}

In this section we provide several new Bernstein functions, which will serve as Laplace exponents in the subordination
representations of Sections \ref{sec:multiterm}--\ref{sec:general}. The following lemma from \cite{Lichenggang2026dual}
is our basic tool.

\begin{lem}[{\cite{Lichenggang2026dual}}]\label{lem:composition-CMF}
    Suppose that functions $f(\cdot),\varphi(\cdot): (0,+\infty)\rightarrow (0,+\infty)$ are infinitely differentiable,
    $f$ is $\mathcal{CMF}$, and $\varphi'(\lambda)=f(\varphi(\lambda))$. Then $\varphi$ is $\mathcal{BF}$.
\end{lem}

\begin{rem}\label{rem:product-CMF}
The same conclusion holds if $\varphi'(\lambda)=l(\lambda)f(\varphi(\lambda))$, where both $l(\cdot)$ and $f(\cdot)$ are $\mathcal{CMF}$.
Indeed, $f\circ\varphi$ is $\mathcal{CMF}$ by the proof of \cite[Lemma 2.4]{Lichenggang2026dual}, hence $\varphi'=l\cdot(f\circ\varphi)$
is a product of two $\mathcal{CMF}$s, which is again $\mathcal{CMF}$ by Proposition \ref{prop:closure} (i).
\end{rem}

The next lemma shows that the Bernstein property is preserved under inversion whenever the derivative of the given
function is a Bernstein function.

\begin{lem}\label{lem:inverse-BF}
Let $F(\cdot): (0,+\infty)\rightarrow (0,+\infty)$ be a bijection of class $C^\infty$ such that $F'$ is $\mathcal{BF}$.
Then the inverse function $\varphi=\varphi(\lambda)=F^{-1}(\lambda)$ is $\mathcal{BF}$.
\end{lem}

\begin{proof}
Since $F'$ is $\mathcal{BF}$, we have $F'>0$ and
\[
\frac1{F'}=(x\mapsto x^{-1})\circ F'\in\mathcal{CMF}
\]
by Proposition \ref{prop:closure} (iii). Moreover, $\varphi=F^{-1}$ satisfies
\[
\varphi'(\lambda)=\frac1{F'(\varphi(\lambda))},
\]
so Lemma \ref{lem:composition-CMF} applies with $f=1/F'$ and proves that $\varphi$ is $\mathcal{BF}$.
\end{proof}

\begin{lem}\label{lem:inverse-polynomial}
Let $m$ be a positive integer and $1\leq\alpha_m<\cdots<\alpha_2<\alpha_1\leq2$, $a_j>0$ ($j=1,2,\cdots, m$).
Then the inverse function $\varphi=\varphi(\lambda)$ ($\lambda>0$) of
\begin{equation}\label{eq:inverse-polynomial}
\lambda=F(\varphi)=\sum\limits_{j=1}^m a_j \varphi^{\alpha_j}, \quad (\varphi>0)
\end{equation}
is $\mathcal{BF}$ with $\varphi(0+)=0$.
\end{lem}
\begin{proof}
It is evident that $F:(0,+\infty)\to(0,+\infty)$ is a $C^\infty$ bijection with $F(0+)=0$, so the inverse function
satisfies $\varphi(0+)=0$. Moreover,
$$
F'(\varphi)=\sum\limits_{j=1}^m a_j\alpha_j \varphi^{\alpha_j-1}>0,
$$
and $F'$ is $\mathcal{BF}$ since $0\leq \alpha_j-1\leq 1$ ($j=1,2,\cdots,m$) and $\mathcal{BF}$ is closed under
addition (Proposition \ref{prop:closure} (ii)). The conclusion follows from Lemma \ref{lem:inverse-BF}.
\end{proof}

\begin{rem}\label{rem:inverse-polynomial-alt}
Lemma \ref{lem:inverse-polynomial} can also be proved via Lemma \ref{lem:composition-CMF}. Indeed,
$$
\varphi'(\lambda)=\frac{1}{\sum\limits_{j=1}^m a_j\alpha_j \varphi^{\alpha_j-1}}:=f(\varphi),
$$
where the denominator is $\mathcal{CBF}$ (each power $\varphi^{\alpha_j-1}$ with $0\leq\alpha_j-1\leq1$ is $\mathcal{CBF}$,
and $\mathcal{CBF}$ is closed under addition), so $f=1/F'$ is a Stieltjes function by Remark \ref{rem:CBF-reciprocal}
and in particular $\mathcal{CMF}$.
\end{rem}

The following theorem provides a new Stieltjes function, which is the key to our main results.

\begin{thm}\label{thm:Stieltjes-quotient}
Let $m\geq1$ and $a_j>0$ ($j=1,2,\cdots,m$). Moreover, let $0<\alpha_{m}<\alpha_{m-1}<\cdots<\alpha_1$ with
$0\leq\alpha_1-\alpha_m\leq 1$. Then
\[
g(\lambda)=\frac{\sum\limits_{j=1}^{m} a_j\lambda^{\alpha_j} }{\sum\limits_{j=1}^{m} a_j\alpha_j\lambda^{\alpha_j} }
\]
is a Stieltjes function and hence is completely monotone on $(0,\infty)$.
\end{thm}

\begin{proof}
When $m=1$, $g(\lambda)=1/\alpha_1$ is a positive constant and the conclusion is trivial. Assume now that $m>1$.
We have
\[
g(\lambda)=\frac{\sum\limits_{j=1}^{m} a_j\lambda^{\alpha_j} }{\sum\limits_{j=1}^{m} a_j\alpha_j\lambda^{\alpha_j} }
  =\frac{\sum\limits_{j=1}^{m} a_j\lambda^{\alpha_j-\alpha_m} }{\sum\limits_{j=1}^{m} a_j\alpha_j\lambda^{\alpha_j-\alpha_m} }.
\]

Let
\[
\Omega=\mathbb{C}\setminus (-\infty,0]
\]
and let
\[
z^\gamma=e^{\gamma \operatorname{Log}z}, \qquad -\pi<\arg z<\pi,
\]
denote the principal branch of the power function.
The functions
\[
P(z)=\sum\limits_{j=1}^{m} a_j z^{\alpha_j-\alpha_m} ,\qquad Q(z)=\sum\limits_{j=1}^{m} a_j\alpha_j z^{\alpha_j-\alpha_m}
\]
are holomorphic on $\Omega$.

We claim that $Q$ has no zeros in $\Omega$. Let $z = \rho e^{i\theta}$ with $\rho > 0$ and $0 < \theta < \pi$. Then
\[
\operatorname{Im} Q(z) = \sum_{j=1}^{m} a_j \alpha_j \rho^{\alpha_j-\alpha_m} \sin((\alpha_j-\alpha_m) \theta).
\]
Because $0 <\alpha_j-\alpha_m\leq 1$ ($j=1,2,\cdots,m-1$) and $0 < \theta < \pi$, one has $0 < (\alpha_j-\alpha_m) \theta < \pi$. Therefore
\[
\operatorname{Im} Q(z) > 0 \quad (\operatorname{Im} z > 0).
\]
By complex conjugation, $\operatorname{Im} Q(z) < 0$ when $\operatorname{Im} z < 0$; on the positive real axis, $Q$ is positive. Hence $Q(z) \neq 0$ throughout $\Omega$, and
\[
\Phi(z) := \frac{P(z)}{Q(z)}
\]
is holomorphic on $\Omega$.

We now determine the sign of $\operatorname{Im}\Phi(z)$ in the upper half-plane. For $z = \rho e^{i\theta}$, where $0 < \theta < \pi$,
\[
|Q(z)|^2 \operatorname{Im}\Phi(z) = \operatorname{Im}(P(z)\overline{Q(z)}).
\]
Expanding the product gives
\begin{align*}
\operatorname{Im}(P(z)\overline{Q(z)}) &= \operatorname{Im}\left( \sum_{i,j=1}^{m} a_i a_j  \alpha_j
                           \rho^{\alpha_i+\alpha_j-2\alpha_m} e^{i(\alpha_i-\alpha_j)\theta} \right) \\
&= \sum_{i,j=1}^{m} a_i a_j  \alpha_j  \rho^{\alpha_i+\alpha_j-2\alpha_m} \sin((\alpha_i-\alpha_j)\theta).
\end{align*}
The diagonal terms vanish. For $i \neq j$, pairing the terms indexed by $(i,j)$ and $(j,i)$ gives
\[
\begin{split}
&a_i a_j \rho^{\alpha_i+\alpha_j-2\alpha_m} \left[ \alpha_j \sin((\alpha_i-\alpha_j)\theta) + \alpha_i \sin((\alpha_j-\alpha_i)\theta) \right]\\
&\qquad = a_i a_j (\alpha_j - \alpha_i) \rho^{\alpha_i+\alpha_j-2\alpha_m} \sin((\alpha_i - \alpha_j)\theta).
\end{split}
\]
Consequently,
\begin{equation}\label{eq:Im-Phi}
\operatorname{Im}\Phi(z) = \frac{\sum_{1\leq i < j \leq m} a_i a_j (\alpha_j - \alpha_i) \rho^{\alpha_i+\alpha_j-2\alpha_m} \sin((\alpha_i-\alpha_j)\theta)}{|Q(z)|^2}.
\end{equation}
For $i < j$, one has $\alpha_j-\alpha_i<0$ and
\[
0 < (\alpha_i - \alpha_j)\theta \leq (\alpha_1-\alpha_m)\theta < \pi,
\]
since $\alpha_1-\alpha_m\leq 1$ and $0<\theta<\pi$; hence $\sin((\alpha_i-\alpha_j)\theta)>0$. It follows that
\[
\operatorname{Im}\Phi(z)<0  \qquad (\operatorname{Im}z>0).
\]

Thus $\Phi$ maps the upper half-plane into the lower half-plane. By complex conjugation, $\Phi$ maps the lower
half-plane into the upper half-plane. Finally,
\[
g(0+)=\lim_{\lambda\to0+}\frac{\sum_{j=1}^{m} a_j\lambda^{\alpha_j-\alpha_m} }{\sum_{j=1}^{m} a_j\alpha_j\lambda^{\alpha_j-\alpha_m} }=\frac{1}{\alpha_m}
\]
exists in $(0,\infty)$. Therefore $g(\lambda)$ is a Stieltjes function by Lemma \ref{lem:Stieltjes-criterion}.
\end{proof}

The following theorem is a generalization of Lemma \ref{lem:inverse-polynomial} and is the core Bernstein result of
this paper.

\begin{thm}\label{thm:inverse-multiterm}
Let $m\geq 2$ and
$$
\lambda=\lambda(\varphi) = \left(\sum\limits_{j=1}^{m-1} a_j\varphi^{\alpha_j} +a_m\varphi^\beta   \right)^{1/\beta}, \quad \varphi>0,
$$
where $0<\alpha_1-\beta\leq 1$, $0<\beta<\alpha_{m-1}<\cdots<\alpha_1$, $a_j>0$ ($j=1,2,\cdots, m$).
Then the inverse function $\varphi=\varphi(\lambda)$ is $\mathcal{BF}$ and $\varphi(0+)=0$.
\end{thm}
\begin{proof}
It is evident that $\lambda=\lambda(\varphi)$ ($\varphi>0$) is strictly increasing with $\lambda(0+)=0$ and
$\lambda(+\infty)=+\infty$, so the inverse function $\varphi=\varphi(\lambda)$ exists and satisfies $\varphi(0+)=0$.
We now show that $\varphi$ is $\mathcal{BF}$.

Differentiating $\lambda^\beta=\sum_{j=1}^{m-1} a_j\varphi^{\alpha_j}+a_m\varphi^{\beta}$ with respect to $\lambda$, we have
\begin{eqnarray*}
\varphi'(\lambda)&=&\frac{\beta \lambda^{\beta-1}}{\sum\limits_{j=1}^{m-1} a_j\alpha_j\varphi^{\alpha_j-1} +a_m\beta \varphi^{\beta-1}} \\
&=&\frac{\lambda^\beta}{\sum\limits_{j=1}^{m-1} a_j\alpha_j\varphi^{\alpha_j} +a_m\beta \varphi^{\beta}}\cdot\frac{\beta\varphi}{\lambda}\\
&=&\frac{\sum\limits_{j=1}^{m-1} a_j\varphi^{\alpha_j} +a_m \varphi^{\beta}}{\sum\limits_{j=1}^{m-1} a_j\alpha_j\varphi^{\alpha_j} +a_m\beta \varphi^{\beta}}
\cdot\frac{\beta\varphi}{\left(\sum\limits_{j=1}^{m-1} a_j\varphi^{\alpha_j} +a_m \varphi^{\beta}\right)^{1/\beta}}\\
&=&\frac{\sum\limits_{j=1}^{m-1} a_j\varphi^{\alpha_j} +a_m \varphi^{\beta}}{\sum\limits_{j=1}^{m-1} a_j\alpha_j\varphi^{\alpha_j} +a_m\beta \varphi^{\beta}}
\cdot\frac{\beta}{\left(\sum\limits_{j=1}^{m-1} a_j\varphi^{\alpha_j-\beta} +a_m \right)^{1/\beta}}\\
&:=& g(\varphi)\cdot f(\varphi).
\end{eqnarray*}
Since the largest order difference appearing in $g$ is $\alpha_1-\beta\leq1$, Theorem \ref{thm:Stieltjes-quotient}
implies that $g(\cdot)$ is a Stieltjes function, and hence $\mathcal{CMF}$.
Furthermore, $\sum_{j=1}^{m-1} a_j\varphi^{\alpha_j-\beta} +a_m$ is $\mathcal{BF}$ since $0<\alpha_j-\beta\leq1$,
and $x\mapsto \beta x^{-1/\beta}$ is $\mathcal{CMF}$, so $f(\cdot)$ is $\mathcal{CMF}$ by the composition property
(Proposition \ref{prop:closure} (iii)).
Hence $gf$ is $\mathcal{CMF}$ by Proposition \ref{prop:closure} (i).
Therefore $\varphi=\varphi(\lambda)$ is $\mathcal{BF}$ by Lemma \ref{lem:composition-CMF}.
\end{proof}

\begin{rem}\label{rem:two-term}
In the case $m=2$, Theorem \ref{thm:inverse-multiterm} can be proved directly. For brevity, consider
$a\varphi^{\alpha}+b\varphi^{\beta}=\lambda^{\beta}$ with $0< \beta < \alpha\leq 2$, $\alpha-\beta\leq 1$,
$a,b>0$. We have
\begin{eqnarray*}
\varphi'(\lambda)&=&\frac{\beta\lambda^{\beta-1}}{a\alpha \varphi^{\alpha-1}+b\beta \varphi^{\beta-1}}
=\frac{\lambda^{\beta}}{a\alpha \varphi^{\alpha}+b\beta \varphi^{\beta}}\cdot\frac{\beta\varphi}{\lambda}\\
&=&\frac{a\varphi^{\alpha}+b\varphi^{\beta}}{a\alpha \varphi^{\alpha}+b\beta \varphi^{\beta}}\cdot\frac{\beta\varphi}{(a\varphi^{\alpha}+b\varphi^{\beta})^{1/\beta}}
=\frac{a\varphi^{\alpha-\beta}+b}{a\alpha \varphi^{\alpha-\beta}+b\beta }\cdot\frac{\beta}{(a\varphi^{\alpha-\beta}+b)^{1/\beta}}\\
&=&\frac{\beta}{\alpha}\cdot \left[1+\frac{b(\alpha-\beta)}{a\alpha \varphi^{\alpha-\beta}+b\beta}\right]\cdot\frac{1}{(a\varphi^{\alpha-\beta}+b)^{1/\beta}}
:=\frac{\beta}{\alpha}\cdot h(\varphi).
\end{eqnarray*}
Then
$G(z)=\left[1+\frac{b(\alpha-\beta)}{a\alpha z+b\beta}\right]\cdot\frac{1}{(az+b)^{1/\beta}}$ is $\mathcal{CMF}$,
and $z=z(\varphi)=\varphi^{\alpha-\beta} $ is $\mathcal{BF}$ for $0<\alpha-\beta\leq 1$ by Proposition
\ref{prop:closure}. Thus $h(\varphi)=G(z(\varphi))$ is $\mathcal{CMF}$
with respect to $\varphi$, and $\varphi(\lambda)$ is $\mathcal{BF}$ by Lemma \ref{lem:composition-CMF}.
\end{rem}

By composing the inverse function in Theorem \ref{thm:inverse-multiterm} with a power function, the condition on
$\beta$ can be relaxed to all $0<\beta\leq\alpha_m$.

\begin{cor}\label{cor:relax-beta}
Let $m\geq 2$, $0<\beta\leq\alpha_m<\cdots<\alpha_1$ with $0<\alpha_1-\alpha_m\leq 1$, and $a_j>0$ ($j=1,2,\cdots, m$).
Then the inverse function $\varphi=\varphi(\lambda)$ ($\lambda>0$) of
\begin{equation}\label{eq:inverse-relax}
\lambda=\lambda(\varphi) = \left(\sum\limits_{j=1}^{m} a_j\varphi^{\alpha_j}    \right)^{1/\beta}, \quad \varphi>0
\end{equation}
is $\mathcal{BF}$ with $\varphi(0+)=0$.
\end{cor}
\begin{proof}
The case $\beta=\alpha_m$ is Theorem \ref{thm:inverse-multiterm}. For general $0<\beta\leq\alpha_m$, define
$$
G(\varphi)=\left(\sum\limits_{j=1}^{m} a_j\varphi^{\alpha_j}\right)^{1/\alpha_m}, \quad \varphi>0.
$$
Its inverse function $G^{-1}$ is $\mathcal{BF}$ with $G^{-1}(0+)=0$ by Theorem \ref{thm:inverse-multiterm}.
Since $\lambda(\varphi)=G(\varphi)^{\alpha_m/\beta}$, the inverse function of \eqref{eq:inverse-relax} is
$$
\varphi(\lambda)=G^{-1}\left(\lambda^{\beta/\alpha_m}\right).
$$
The power function $\lambda\mapsto\lambda^{\beta/\alpha_m}$ is $\mathcal{BF}$ because $0<\beta/\alpha_m\leq1$,
and the composition of two Bernstein functions is again $\mathcal{BF}$ (Proposition \ref{prop:closure} (ii)).
Clearly, $\varphi(0+)=0$.
\end{proof}

\begin{cor}\label{cor:attenuated}
Let $a,b>0$ and $0<\beta<\alpha\leq 2$. Then the inverse function $\varphi=\varphi(\lambda)$ of
$$
a^2\varphi^\alpha+2ab\varphi^{(\alpha+\beta)/2}+b^2\varphi^\beta=\lambda^\beta
$$
is $\mathcal{BF}$ with $\varphi(0+)=0$.
\end{cor}

\begin{proof}
We have
$$
(a\varphi^{\alpha/2}+b\varphi^{\beta/2})^2=\lambda^\beta,
$$
hence
$$
a\varphi^{\alpha/2}+b\varphi^{\beta/2}=\lambda^{\beta/2}.
$$
This is the two-term case of Theorem \ref{thm:inverse-multiterm} with orders $\alpha/2$, $\beta/2$ and target power
$\beta/2$, where $\alpha/2-\beta/2=(\alpha-\beta)/2\leq 1$ is automatically satisfied since $\alpha\leq2$.
Therefore $\varphi=\varphi(\lambda)$ is $\mathcal{BF}$ with $\varphi(0+)=0$.
\end{proof}

Another useful observation is: when two characteristics share the same outer symbol, the
Bernstein property of $F^{-1}\circ G$ is independent of that outer symbol.

\begin{prop}\label{prop:shared-outer}
Let $f(\cdot)$, $g(\cdot)$ and $Q(\cdot)$ be strictly increasing bijections of $(0,\infty)$ with
$f(0+)=g(0+)=0$, and set
$$
F(\lambda)=Q(f(\lambda)),\qquad G(\lambda)=Q(g(\lambda)),\qquad \lambda>0.
$$
Then $F$ and $G$ are strictly increasing bijections of $(0,\infty)$, and
\begin{equation}\label{eq:shared-outer}
F^{-1}\circ G=f^{-1}\circ g.
\end{equation}
In particular, $F^{-1}\circ G$ is $\mathcal{BF}$ if and only if $f^{-1}\circ g$ is $\mathcal{BF}$.
\end{prop}
\begin{proof}
Since $Q$ is a bijection of $(0,\infty)$, so are $F$ and $G$, and $F^{-1}=f^{-1}\circ Q^{-1}$.
Substituting this into $F^{-1}\circ G$ yields \eqref{eq:shared-outer}.
\end{proof}

\begin{rem}\label{rem:shared-outer-telegraph}
Proposition \ref{prop:shared-outer} applies in particular to the characteristics of telegraph-type
equations: for $0<\beta\leq\alpha$ and $a>0$,
$$
F(\lambda)=\lambda^{2\alpha}+2a\lambda^{\alpha}=Q(\lambda^{\alpha}),\qquad
G(\lambda)=\lambda^{2\beta}+2a\lambda^{\beta}=Q(\lambda^{\beta}),\qquad
Q(\mu)=\mu^2+2a\mu,
$$
hence $F^{-1}\circ G(\lambda)=\lambda^{\beta/\alpha}$ is $\mathcal{BF}$ for all $0<\beta\leq\alpha$,
with no restriction on the gap between the highest order $2\beta$ of $G$ and the lowest order $\alpha$
of $F$. This observation is used essentially in Example \ref{ex:fractional-telegraph}.
\end{rem}

Next, we turn to Bernstein functions related to tempered stable processes.

\begin{prop}\label{thm:CBF-tempered-power}
Let $0<\beta\leq 1$, $h>0$, $\gamma>0$ and $0<\beta\gamma\leq 1$. Then the function
$\varphi(\lambda)=(\lambda^\beta+h)^{\gamma}$ ($\lambda>0$) is $\mathcal{CBF}$.
\end{prop}

\begin{proof}
The limit
$$
\varphi(0+)=\lim_{(0,\infty)\ni\lambda\rightarrow 0}\varphi(\lambda)=h^\gamma
$$
clearly exists.

Let $\operatorname{Log}$ denote the principal branch of the logarithm on $\mathbb C\setminus(-\infty,0]$. Define
\[
F(z)=\bigl(z^\beta+h\bigr)^\gamma
= \exp\Bigl(\gamma \operatorname{Log}\bigl(z^\beta+h\bigr)\Bigr),
\qquad z\in\mathbb C\setminus(-\infty,0],
\]
where $z^\beta:=\exp(\beta\operatorname{Log} z)$.

First, $F$ is well-defined and analytic on $\mathbb C\setminus(-\infty,0]$. Indeed, for $z\notin(-\infty,0]$, write $\arg z=\theta\in(-\pi,\pi)$. Then
\[
\arg(z^\beta)=\beta\theta\in(-\beta\pi,\beta\pi)\subset(-\pi,\pi),
\]
so $z^\beta\notin(-\infty,0]$. Since $h> 0$, the sum $z^\beta+h$ also avoids $(-\infty,0]$. Hence the logarithm is defined, and $F$ is analytic. For $\lambda>0$, $F(\lambda)=\varphi(\lambda)$, so $F$ is the required extension.

Now let $\operatorname{Im} z>0$ and set $\theta=\arg z\in(0,\pi)$. Then
\[
\arg(z^\beta)=\beta\theta\in(0,\beta\pi)\subset(0,\pi).
\]
Adding $h> 0$ gives
\[
0<\arg(z^\beta+h)\le \beta\theta.
\]
Therefore
\[
\arg F(z)=\gamma\arg(z^\beta+h)\le \beta\gamma\,\theta < \beta\gamma\,\pi \le \pi,
\]
because $\theta<\pi$ and $\beta\gamma\le 1$. Also $\arg F(z)>0$. Thus
\[
0<\arg F(z)<\pi,
\]
so $\operatorname{Im} F(z)>0$ for all $\operatorname{Im} z>0$. By Proposition \ref{prop:CBF-characterization}, $\varphi$ is a complete Bernstein function.
\end{proof}

\begin{rem}\label{rem:BGT-open-question}
Proposition \ref{thm:CBF-tempered-power} is closely related to the operation $f\mapsto h_\alpha$,
$h_\alpha(z)=f(z^\alpha)^{1/\alpha}$, on Bernstein functions studied by Batty, Gomilko and Tomilov
\cite{BattyGomilkoTomilov2017}: it is known that $h_\alpha\in\mathcal{BF}$ for every $f\in\mathcal{BF}$ when
$\alpha\in(0,1/2]$ \cite[Proposition 3.3(2)]{BattyGomilkoTomilov2017}, that the case $\alpha\in(1/2,1)$ for
general $f\in\mathcal{BF}$ remains open, and that $h_\alpha\in\mathcal{CBF}$ for all $\alpha\in(0,1)$ whenever
$f\in\mathcal{CBF}$ \cite[Proposition 3.11(2)]{BattyGomilkoTomilov2017}. Indeed, Proposition
\ref{thm:CBF-tempered-power} can be deduced from the latter result: with $\gamma=1/\beta$ one has
$(\lambda^\beta+h)^{1/\beta}=f(\lambda^\beta)^{1/\beta}$ for $f(w)=w+h\in\mathcal{CBF}$, and the case
$\beta\gamma<1$ follows by further composing with the $\mathcal{CBF}$ function $x\mapsto x^{\beta\gamma}$.
For the reader's convenience, we have included the short self-contained proof above.
\end{rem}

Note that the exponent $\gamma$ in Proposition \ref{thm:CBF-tempered-power} is allowed to be greater than $1$, in which
case the conclusion cannot be obtained by composing $\lambda^\beta$ with $x^\gamma$ directly. The next proposition
shows that, for the shifted tempered function $[(\lambda+b)^\beta-b^\beta]^\gamma$, the restriction $\gamma\leq1$
becomes necessary.

\begin{prop}\label{prop:tempered-iff}
Let $b>0$, $0<\beta<1$, $\gamma>0$, and
$\psi(\lambda)=[(\lambda+b)^\beta-b^\beta]^\gamma$.
Then $\psi(\lambda)$ ($\lambda>0$) is $\mathcal{BF}$ if and only if $0<\gamma\leq1$.
\end{prop}
\begin{proof}
If $0<\gamma\leq1$, then $\psi=g\circ f$ with $f(\lambda)=(\lambda+b)^\beta-b^\beta$ and $g(x)=x^\gamma$.
Both $f$ and $g$ are $\mathcal{BF}$, hence $\psi$ is $\mathcal{BF}$ by Proposition \ref{prop:closure} (ii).

For the necessity, we compute
\[
\psi'(\lambda)=\beta\gamma[(\lambda+b)^\beta-b^\beta]^{\gamma-1} (\lambda+b)^{\beta-1}
\]
and
\[
\psi''(\lambda)=\beta\gamma[(\lambda+b)^\beta-b^\beta]^{\gamma-2}(\lambda+b)^{\beta-2}\cdot \left[(\beta\gamma-1) (\lambda+b)^{\beta}+(1-\beta)b^\beta\right].
\]
When $\gamma>1$, we obtain
\[
\lim\limits_{\lambda\to 0+}\left[(\beta\gamma-1) (\lambda+b)^{\beta}+(1-\beta)b^\beta\right] = (\beta\gamma-1) b^{\beta}+(1-\beta)b^\beta =\beta(\gamma-1) b^{\beta}>0.
\]
By continuity,  $\psi''(\lambda)>0$ for small $\lambda>0$, which contradicts the complete monotonicity of $\psi'$.
Therefore $\psi(\lambda)$ is not  $\mathcal{BF}$  when $\gamma>1$.
\end{proof}

Proposition \ref{prop:tempered-iff} shows that for the purely shifted power
$[(\lambda+b)^\beta-b^\beta]^\gamma$ the restriction $\gamma\leq1$ is necessary. The next result
shows that an additive constant inside the outer power changes the picture: an outer power
$1/\alpha>1$ becomes admissible precisely when the constant is large enough.

\begin{prop}\label{prop:shifted-power-iff}
Let $a,b>0$ and $0<\beta<\alpha<1$. Define
\[
\varphi(\lambda)=\bigl[(\lambda+b)^\beta+a^\alpha-b^\beta\bigr]^{1/\alpha}-a,
\qquad \lambda>0.
\]
Then $\varphi$ is $\mathcal{BF}$ if and only if $a^\alpha\geq b^\beta$.
\end{prop}
\begin{proof}
Set $c=a^\alpha-b^\beta$ and note that $\varphi(0+)=0$ in any case.

\textit{Sufficiency.} Assume $c\geq0$ and let $\psi(\mu)=(\mu^\beta+c)^{1/\alpha}$, $\mu>0$.
If $c>0$, then $\psi$ is $\mathcal{CBF}$ by Proposition \ref{thm:CBF-tempered-power} applied with
$\gamma=1/\alpha$, since $\beta\cdot(1/\alpha)<1$; if $c=0$, then $\psi(\mu)=\mu^{\beta/\alpha}$ is
obviously $\mathcal{CBF}$. Write $\varphi(\lambda)=\psi(\lambda+b)-\psi(b)$. The shift
$\mu\mapsto\mu+b$ preserves complete monotonicity: by Proposition \ref{prop:LK-representation} (1),
$\psi'(\mu)=\int_{[0,\infty)}e^{-\mu t}\nu(dt)$ for a nonnegative measure $\nu$, hence
\[
\varphi'(\lambda)=\psi'(\lambda+b)=\int_{[0,\infty)}e^{-\lambda t}\,e^{-bt}\nu(dt),
\qquad \lambda>0,
\]
is again $\mathcal{CMF}$. Since $\varphi'\geq0$ and $\varphi(0+)=\psi(b)-\psi(b)=0$, we conclude
that $\varphi$ is $\mathcal{BF}$.

\textit{Necessity.} Assume $a^\alpha<b^\beta$ and set
\[
d=b^\beta-a^\alpha>0,\qquad r=b-d^{1/\beta}>0,\qquad q=\frac1\alpha-1>0.
\]
The derivative is
\[
h(z)=\varphi'(z)=\frac{\beta}{\alpha}(z+b)^{\beta-1}\bigl((z+b)^\beta-d\bigr)^q .
\]
With principal branches, $h$ is holomorphic on $\Omega_r=\mathbb C\setminus(-\infty,-r]$.
Indeed, the cut of $z+b$ is contained in the excluded ray; for nonreal $z$ in this domain,
$(z+b)^\beta$ has imaginary part of the same sign as $z$, so subtracting the positive real number
$d$ does not move the inner expression onto the cut of the outer power; and for real $z>-r$ the
inner expression is strictly positive.

Fix $x_0>0$. Since $h$ is holomorphic on $\Omega_r$ and the disk $|z-x_0|<x_0+r$ is contained in
$\Omega_r$, the Taylor series of $h$ around $x_0$ converges throughout this disk. If $\varphi$ were
$\mathcal{BF}$, then $h=\varphi'$ would be $\mathcal{CMF}$, that is, $(-1)^n h^{(n)}(x_0)\geq0$ for
all $n\geq0$, and for every real $y\in(-r,x_0)$ we would obtain
\[
h(y)=\sum_{n=0}^{\infty}\frac{(-1)^n h^{(n)}(x_0)}{n!}(x_0-y)^n\geq h(x_0)>0.
\]
On the other hand, as $y\downarrow-r$, the factor $(y+b)^{\beta-1}$ tends to the finite positive
value $(b-r)^{\beta-1}$, while $\bigl((y+b)^\beta-d\bigr)^q$ tends to zero because $q>0$. Thus
$h(y)\to0$, a contradiction. Therefore $\varphi$ is not $\mathcal{BF}$ when $a^\alpha<b^\beta$.
\end{proof}

\section{Subordination from multi-term to single-term TFDEs}\label{sec:multiterm}

Now we prove our main result on the subordination representation from the multi-term equation \eqref{eq:MTFDE} to
the single-term equation \eqref{eq:STFDE}.

\begin{thm}\label{thm:MTFDE-to-STFDE}
Assume $m\geq 2$, $0<\beta\leq\alpha_{m}<\cdots<\alpha_1\leq 2$, $0<\alpha_1-\alpha_m\leq 1$, $a_j>0$ ($j=1,2,\cdots, m$).
Let $\varphi(\cdot)$ be the inverse function of
$$
\lambda=\lambda(\varphi) = \left(\sum\limits_{j=1}^{m} a_j\varphi^{\alpha_j}    \right)^{1/\beta}, \quad \varphi>0.
$$
Moreover, suppose that the multi-term TFDE
\begin{equation}\label{eq:MTFDE-thm}
\begin{aligned}
\begin{cases}
    \sum\limits_{j=1}^m a_j D_t^{\alpha_j}u(t)=Au(t), \quad\quad t>0;\\
     u(0)=\psi, \quad u'(0)=0 \ (\text{if } \alpha_1>1)
\end{cases}
\end{aligned}
\end{equation}
admits a bounded strong solution $u(t)$, and that the single-term TFDE
\begin{equation}\label{eq:STFDE-thm}
\begin{aligned}
\begin{cases}
    D_t^\beta v(t)=Av(t),\quad\quad t>0; \\
    v(0) = \psi, \quad v'(0)=0\ (\text{if }\beta>1)
\end{cases}
\end{aligned}
\end{equation}
admits an exponentially bounded strong solution $v(t)$. Then $v(t)$ can be represented as
\begin{equation}\label{eq:subordination-MTFDE}
v(t)=\int_0^\infty p(s,t)u(s)ds,\quad t>0,
\end{equation}
where $p(s,t)$ ($s>0,t>0$) is a probability density with respect to $s$ for each $t>0$, and satisfies
\begin{equation}\label{eq:p-Laplace}
\int_0^\infty e^{-\lambda t}p(s,t)dt=\frac{\varphi(\lambda)}{\lambda}e^{-s \varphi(\lambda)}.
\end{equation}
\end{thm}

\begin{proof}
We divide the proof into three steps.

{\bf Step 1.}
By Corollary \ref{cor:relax-beta}, the inverse function $\varphi=\varphi(\lambda)$ is $\mathcal{BF}$ with
$\varphi(0+)=0$, so there exists a subordinator $\mathbf{D}(t)$ with Laplace exponent $\varphi(\cdot)$
\cite{Applebaum2009,Schilling2012}.
It follows from
\[
\lambda^\beta=\sum\limits_{j=1}^{m} a_j\varphi^{\alpha_j}
\]
that
\[
1=\sum\limits_{j=1}^{m} a_j \varphi^{\alpha_j-\beta}   \cdot \left(\frac{\varphi}{\lambda}\right)^\beta.
\]
Thus
\[
\lim\limits_{\lambda\to\infty} \frac{\varphi(\lambda)}{\lambda}=0,
\]
since $m\geq2$, $a_j>0$, $\alpha_j-\beta\geq0$ and $\alpha_1-\beta>0$. Therefore $\mathbf{D}(t)$ has no drift.
Moreover, $\varphi(+\infty)=+\infty$, and since $\mathbf{D}(t)$ has no drift, the L{\'e}vy measure of $\varphi(\cdot)$
is an infinite measure.

Define the inverse subordinator \cite{Meerschaert2012book}
$$
\mathbf{E}(t)=\inf\{z>0:\mathbf{D} (z)>t\}.
$$
Then, by Remark \ref{rem:density-allt}, $\mathbf{E}(t)$ has a Lebesgue density $p(s,t)$ for every $t>0$,
whose Laplace transform with respect to $t$ is given by \eqref{eq:p-Laplace}; see Proposition
\ref{prop:inverse-subordinator-Laplace}.

{\bf Step 2.}
Taking the Laplace transform on both sides of \eqref{eq:MTFDE-thm} and using \eqref{eq:Caputo-Laplace}, we obtain
$$
\sum\limits_{j=1}^m a_j \left(\varphi^{\alpha_j}\widehat{u}(\varphi)-\varphi^{\alpha_j-1} \psi  \right)=A\widehat{u}(\varphi),
$$
and therefore
\begin{equation}\label{eq:u-hat}
\widehat{u}(\varphi)=\frac{\sum\limits_{j=1}^m a_j \varphi^{\alpha_j}}{\varphi}\left( \sum\limits_{j=1}^m a_j \varphi^{\alpha_j}-A  \right)^{-1} \psi.
\end{equation}

Similarly, taking the Laplace transform on both sides of \eqref{eq:STFDE-thm}, we have
$$
\widehat{v}(\lambda)=\lambda^{\beta-1} \left(\lambda^{\beta} -A  \right)^{-1}\psi.
$$
(Here we use different letters $\varphi>0$ and $\lambda>0$ for the variables of $\widehat{u}$ and $\widehat{v}$.)

{\bf Step 3.}
Consequently,
\begin{eqnarray*}
\widehat{v}(\lambda) &=&\lambda^{\beta-1} \left(\lambda^{\beta} -A  \right)^{-1}\psi\\
&=&  \frac{\sum\limits_{j=1}^m a_j (\varphi(\lambda))^{\alpha_j}}{\lambda}
   \cdot \left[\sum\limits_{j=1}^m a_j (\varphi(\lambda))^{\alpha_j}-A \right]^{-1}  \psi \\
&=&\frac{\varphi(\lambda)}{\lambda} \cdot  \frac{\sum\limits_{j=1}^m a_j (\varphi(\lambda))^{\alpha_j}}{\varphi(\lambda)}
   \cdot \left[\sum\limits_{j=1}^m a_j (\varphi(\lambda))^{\alpha_j}-A \right]^{-1}  \psi    \\
&=&\int_0^\infty \frac{\varphi(\lambda)}{\lambda}e^{-s \varphi(\lambda)}u(s)ds     \qquad (\text{using } \eqref{eq:u-hat})\\
&=&\int_0^\infty \left(  \int_0^\infty e^{-\lambda t} p(s,t)dt  \right)u(s)ds \qquad (\text{using } \eqref{eq:p-Laplace}) \\
&=&\int_0^\infty e^{-\lambda t}\left( \int_0^\infty p(s,t)u(s)ds\right)dt,
\end{eqnarray*}
where the exchange of the order of integration in the last equality is justified by Fubini's theorem,
since $u$ is bounded and $p(\cdot,t)$ is a probability density.
The representation \eqref{eq:subordination-MTFDE} now follows from the uniqueness of the Laplace
transform, which a priori gives the equality for almost every $t>0$, together with the
right-continuity in $t$ of both sides: the right-hand side of \eqref{eq:subordination-MTFDE} equals
$\mathbb{E}[u(\mathbf{E}(t))]$, which is right-continuous in $t$ since $u$ is bounded and continuous
and the sample paths of $\mathbf{E}(\cdot)$ are right-continuous.
\end{proof}

\begin{rem}\label{rem:STFDE-existence}
The existence of the strong solution $v$ assumed in Theorem \ref{thm:MTFDE-to-STFDE} is automatic under
standard generation assumptions on $A$: if $A$ generates a bounded $C_0$-semigroup (for $0<\beta\leq1$) or
a bounded cosine family (for $1<\beta<2$), then \eqref{eq:STFDE-thm} is well-posed with a bounded solution;
see Subsection \ref{subsec:prelim} \cite{Bajlekova2001}.
\end{rem}

\begin{rem}\label{rem:two-directions}
A crucial condition for the subordination from \eqref{eq:STFDE-thm} to
\eqref{eq:MTFDE-thm} is $\beta\geq \alpha_1$ (see \cite{Bajlekova2018}),
while that for the reverse direction established here (from \eqref{eq:MTFDE-thm} to \eqref{eq:STFDE-thm}) is
$\beta\leq \alpha_m$. The two directions are therefore complementary.
\end{rem}

\begin{example}\label{ex:telegraph-diffusion}
Consider
\begin{equation*}
\begin{aligned}
\begin{cases}
    u''(t)+2a u'(t)=Au(t), \quad \quad t>0,\\
    u(0)=u_0, \quad u'(0)=0
\end{cases}
\end{aligned}
\end{equation*}
and
\begin{equation*}
\begin{aligned}
\begin{cases}
    v'(t) = Av(t),\quad\quad t>0, \\
    v(0) = u_0,
\end{cases}
\end{aligned}
\end{equation*}
where $a>0$.
Then we have
$
v(t)=\int_0^\infty p(s,t)u(s)ds
$ ($t>0$),
where $p(s,t)$ is the probability density of an inverse subordinator $\mathbf{E}(t)$ with Laplace exponent
$\varphi(\lambda)=-a+\sqrt{a^2+\lambda}$ (note that $\varphi(\lambda)\sim\lambda^{1/2}$ as
$\lambda\to\infty$, so $\varphi$ is driftless with $\varphi(\infty)=\infty$);
see \cite{Lichenggang2026dual}. This is the case $m=2$, $\alpha_1=2$, $\alpha_2=\beta=1$ of Theorem
\ref{thm:MTFDE-to-STFDE}.
\end{example}

\begin{example}\label{ex:attenuated-wave}
 Let $a,b>0$. Consider the time-fractional attenuated wave equation \cite{Straka2013}
\begin{equation*}
\begin{aligned}
\begin{cases}
    a^2D_t^{2}u(t)+2ab D_t^{1+\alpha}u(t)+b^2 D_t^{2\alpha}u(t) =Au(t),\quad\quad t>0,\ 0<\alpha<1;\\
    u(0) = \psi, \quad u'(0)=0.
\end{cases}
\end{aligned}
\end{equation*}
The orders are $2$, $1+\alpha=\frac{2+2\alpha}{2}$ and $2\alpha$, so Corollary \ref{cor:attenuated} applies with
$(\alpha,\beta)$ replaced by $(2,2\alpha)$: the inverse function $\varphi(\lambda)$ of
$a\varphi+b\varphi^{\alpha}=\lambda^{\alpha}$
is $\mathcal{BF}$ with $\varphi(0+)=0$; moreover, $\varphi(\lambda)\sim\lambda^{\alpha}/a$ as
$\lambda\to\infty$ since $0<\alpha<1$, so $\varphi$ is driftless with $\varphi(\infty)=\infty$.
As a result, arguing as in the proof of Theorem \ref{thm:MTFDE-to-STFDE} with Corollary
\ref{cor:attenuated} in place of Corollary \ref{cor:relax-beta} (so that no gap condition on the
orders $2,1+\alpha,2\alpha$ is needed), the solution of
\begin{equation*}
\begin{aligned}
\begin{cases}
    D_t^{2\alpha}v(t)=Av(t),\quad\quad t>0; \\
    v(0) = \psi, \quad v'(0)=0\ (\text{if }\alpha>1/2)
\end{cases}
\end{aligned}
\end{equation*}
can be represented as
$
v(t)=\int_0^\infty p(s,t)u(s)ds,
$
where $p(s,t)$ is the density of the inverse subordinator with Laplace exponent $\varphi(\cdot)$.
\end{example}

The two complementary directions can be chained together: first from a multi-term equation to a single-term
equation (Theorem \ref{thm:MTFDE-to-STFDE}), and then from this single-term equation to another multi-term
equation (the known direction; see \cite{Bajlekova2018} and item (4) in the introduction).
This yields a subordination representation between two multi-term TFDEs.

\begin{thm}\label{thm:multi-to-multi}
Consider the two multi-term TFDEs
\begin{equation}\label{eq:MTFDE-source}
\sum\limits_{j=1}^m a_j D_t^{\alpha_j}u(t)=Au(t),\quad t>0;
\qquad\quad
\sum\limits_{i=1}^n b_i D_t^{\gamma_i}w(t)=Aw(t),\quad t>0,
\end{equation}
with the initial conditions $u(0)=w(0)=\psi$ {\rm(}$u'(0)=0$ if $\alpha_1>1$, and $w'(0)=0$ if
$\gamma_1>1${\rm)}, where $m,n\geq 2$, $a_j,b_i>0$, and the orders satisfy
\begin{gather*}
0<\alpha_m<\cdots<\alpha_1\leq 2,\quad 0<\alpha_1-\alpha_m\leq 1;\qquad
0<\gamma_n<\cdots<\gamma_1\leq 2,\quad 0<\gamma_1-\gamma_n\leq 1;\\
\gamma_1\leq \alpha_m.
\end{gather*}
Assume that the first equation in \eqref{eq:MTFDE-source} admits a bounded strong solution $u(t)$ and that the
second equation admits a bounded strong solution $w(t)$.
Then $w(t)$ can be represented as
\begin{equation}\label{eq:subordination-multi-multi}
w(t)=\int_0^\infty p_0(s,t)u(s)ds,\quad t>0,
\end{equation}
where
\begin{equation}\label{eq:p0-convolution}
p_0(s,t)=\int_0^\infty p_1(s,\tau)P_2(d\tau,t),
\end{equation}
$p_1(s,t)$ is the probability density of the inverse subordinator with Laplace exponent
$\varphi_1(\cdot)$, and $P_2(ds,t)$ is the distribution of the inverse subordinator with Laplace exponent
$\varphi_2(\cdot)$: for any fixed $\beta\in[\gamma_1,\alpha_m]$,
$\varphi_1(\cdot)$ is the inverse function of
$\lambda=\bigl(\sum_{j=1}^{m} a_j\varphi^{\alpha_j}\bigr)^{1/\beta}$ ($\varphi>0$), and
$\varphi_2(\lambda)=\bigl(\sum_{i=1}^{n} b_i\lambda^{\gamma_i}\bigr)^{1/\beta}$.
Moreover, writing $F(\lambda)=\sum\limits_{j=1}^m a_j\lambda^{\alpha_j}$ and
$G(\lambda)=\sum\limits_{i=1}^n b_i\lambda^{\gamma_i}$, the composition
\begin{equation}\label{eq:composed-exponent}
\varphi_1\circ\varphi_2=F^{-1}\circ G
\end{equation}
is $\mathcal{BF}$ with $(\varphi_1\circ\varphi_2)(0+)=0$, and $p_0(s,t)$ is the probability density of the
inverse subordinator with Laplace exponent $F^{-1}\circ G$:
\begin{equation}\label{eq:p0-Laplace}
\int_0^\infty e^{-\lambda t}p_0(s,t)dt
=\frac{\varphi_1(\varphi_2(\lambda))}{\lambda}\,e^{-s\,\varphi_1(\varphi_2(\lambda))}.
\end{equation}
\end{thm}
\begin{proof}
Fix any $\beta\in[\gamma_1,\alpha_m]$; it serves only as an analytic device for the decomposition
\eqref{eq:p0-convolution}. Since $\beta\leq\alpha_m$, Corollary
\ref{cor:relax-beta} implies that $\varphi_1$ is $\mathcal{BF}$ with $\varphi_1(0+)=0$; since
$\beta\geq\gamma_1$, the function $\varphi_2$ is $\mathcal{BF}$ with $\varphi_2(0+)=0$
(see \cite{Bajlekova2018} and item (4) in the introduction). As the composition of two Bernstein
functions, $\varphi_1\circ\varphi_2$ is $\mathcal{BF}$ (Proposition \ref{prop:closure} (ii)), and clearly
$(\varphi_1\circ\varphi_2)(0+)=0$. Moreover, $F(\varphi_1(\mu))=\mu^\beta$ by the definition of
$\varphi_1$, and $\varphi_2(\lambda)^\beta=G(\lambda)$ by the definition of $\varphi_2$; hence
$F(\varphi_1(\varphi_2(\lambda)))=G(\lambda)$, that is, \eqref{eq:composed-exponent} holds.
Since $\varphi_1$ has no drift (see Step 1 of the proof of Theorem \ref{thm:MTFDE-to-STFDE}), neither does
$\varphi_1\circ\varphi_2$, and $\varphi_1(\varphi_2(+\infty))=+\infty$; therefore, by Remark
\ref{rem:density-allt}, the inverse subordinator with Laplace exponent $F^{-1}\circ G$ admits a
probability density for every $t>0$, which satisfies \eqref{eq:p0-Laplace} by Proposition
\ref{prop:inverse-subordinator-Laplace}. This density coincides with
$p_0(s,t)$ defined by \eqref{eq:p0-convolution}: taking the Laplace transform in $t$,
using the identity
$\int_0^\infty e^{-\lambda t}P_2(d\tau,t)dt=\frac{\varphi_2(\lambda)}{\lambda}e^{-\tau\varphi_2(\lambda)}$
and then \eqref{eq:p-Laplace} for $p_1$, we get
\begin{align*}
\int_0^\infty e^{-\lambda t}p_0(s,t)dt
&=\int_0^\infty \frac{\varphi_2(\lambda)}{\lambda}e^{-\tau\varphi_2(\lambda)}p_1(s,\tau)d\tau
=\frac{\varphi_1(\varphi_2(\lambda))}{\lambda}\,e^{-s\,\varphi_1(\varphi_2(\lambda))},
\end{align*}
which is \eqref{eq:p0-Laplace}.

It remains to verify \eqref{eq:subordination-multi-multi}. Taking the Laplace transform of both equations
in \eqref{eq:MTFDE-source} (via \eqref{eq:Caputo-Laplace} and the initial conditions), we obtain
$\widehat{u}(\mu)=\frac{F(\mu)}{\mu}\bigl[F(\mu)-A\bigr]^{-1}\psi$ and
$\widehat{w}(\lambda)=\frac{G(\lambda)}{\lambda}\bigl[G(\lambda)-A\bigr]^{-1}\psi$.
Taking the Laplace transform in $t$ of the right-hand side of \eqref{eq:subordination-multi-multi} and
using \eqref{eq:p0-Laplace} and Fubini's theorem, we obtain
\begin{align*}
\int_0^\infty e^{-\lambda t}\left(\int_0^\infty p_0(s,t)u(s)ds\right)dt
&=\int_0^\infty \frac{\varphi_1(\varphi_2(\lambda))}{\lambda}\,
  e^{-s\,\varphi_1(\varphi_2(\lambda))}u(s)\,ds\\
&=\frac{\varphi_1(\varphi_2(\lambda))}{\lambda}\,\widehat{u}\bigl(\varphi_1(\varphi_2(\lambda))\bigr)
=\frac{G(\lambda)}{\lambda}\bigl[G(\lambda)-A\bigr]^{-1}\psi
=\widehat{w}(\lambda),
\end{align*}
where the third equality uses the above formula for $\widehat{u}$ together with \eqref{eq:composed-exponent}.
The representation \eqref{eq:subordination-multi-multi} now follows from the uniqueness of the Laplace
transform and the right-continuity in $t$ of both sides, as in the proof of Theorem
\ref{thm:MTFDE-to-STFDE}.
\end{proof}

\begin{rem}\label{rem:multi-multi-existence}
If $A$ generates a bounded $C_0$-semigroup (when all the orders involved are at most $1$) or a bounded
cosine family (in general), then both equations in \eqref{eq:MTFDE-source} are well-posed with bounded
strong solutions, so the existence of $u$ and $w$ assumed in Theorem \ref{thm:multi-to-multi} is
automatic: the corresponding single-term equations are well-posed by \cite{Bajlekova2001}, and the
well-posedness of the two multi-term equations then follows by \cite{Lichenggang2012}.
\end{rem}

\begin{rem}\label{rem:p2-density}
The intermediate distribution $P_2(ds,t)$ has a Lebesgue density $p_2(s,t)$ for every $t>0$ whenever
$\beta>\gamma_1$: indeed, in this case $\varphi_2(\lambda)\sim b_1^{1/\beta}\lambda^{\gamma_1/\beta}$ with
$\gamma_1/\beta<1$, so $\varphi_2$ is driftless and Remark \ref{rem:density-allt} applies.
The boundary case $\beta=\gamma_1$ (which is enforced when $\gamma_1=\alpha_m$) is more delicate: then
\[
\varphi_2(\lambda)=b_1^{1/\gamma_1}\lambda\Big(1+\sum\limits_{i=2}^n
\frac{b_i}{b_1}\lambda^{\gamma_i-\gamma_1}\Big)^{1/\gamma_1}
\]
has the positive drift $b_1^{1/\gamma_1}=\lim_{\lambda\to\infty}\varphi_2(\lambda)/\lambda>0$. Since
$(1+u)^{1/\gamma_1}-1\geq \frac{1}{2\gamma_1}u$ for all sufficiently small $u>0$, we have
\[
\sigma_2(0,\infty)=\lim_{\lambda\to\infty}\big(\varphi_2(\lambda)-b_1^{1/\gamma_1}\lambda\big)
\geq \frac{b_1^{1/\gamma_1-1}}{2\gamma_1}\lim_{\lambda\to\infty}\sum\limits_{i=2}^n
b_i\lambda^{1-(\gamma_1-\gamma_i)},
\]
where $\sigma_2$ denotes the L{\'e}vy measure of $\varphi_2$. Hence $\sigma_2$ is infinite if
$\gamma_1-\gamma_2<1$, in which case $\mathbf{E}_{\varphi_2}(t)$ has no atom (cf.\
\cite[Example 3.4]{Meerschaert2008} for an explicit density formula in the stable case). On the other hand, if $\gamma_1-\gamma_2=1$, then necessarily $n=2$ (because
$\gamma_1-\gamma_n\leq1$), and
$\sigma_2(0,\infty)=b_1^{1/\gamma_1-1}b_2/\gamma_1<\infty$: the jump part of $\mathbf{D}_{\varphi_2}$ is
then a compound Poisson process, and on the event that it has no jump on $[0,b_1^{-1/\gamma_1}t]$ -- an
event of probability $\exp\big(-\frac{b_2}{b_1\gamma_1}t\big)>0$ -- one has
$\mathbf{E}_{\varphi_2}(t)=b_1^{-1/\gamma_1}t$. Thus $\mathbf{E}_{\varphi_2}(t)$ has an atom at
$s=b_1^{-1/\gamma_1}t$ and admits no Lebesgue density. In all cases, however, the composed exponent
$\varphi_1\circ\varphi_2=F^{-1}\circ G$ is driftless with an infinite L{\'e}vy measure, so the final
kernel $p_0(s,t)$ in \eqref{eq:p0-convolution} is a genuine probability density for every $t>0$.
\end{rem}

\begin{rem}\label{rem:multi-multi-stochastic}
Although an intermediate order $\beta\in[\gamma_1,\alpha_m]$ is used in the construction, the resulting
kernel $p_0(s,t)$ does not depend on the choice of $\beta$, since its Laplace exponent
$F^{-1}\circ G$ involves only the characteristics of the two multi-term equations.
Moreover, \eqref{eq:p0-convolution} admits a clear probabilistic meaning: if $\mathbf{D}_1(t)$ and
$\mathbf{D}_2(t)$ are independent subordinators with Laplace exponents $\varphi_1$ and $\varphi_2$, and
$\mathbf{E}_1(t)$, $\mathbf{E}_2(t)$ are the corresponding inverse subordinators, then $p_0(s,t)$ is the
probability density of the composed process $\mathbf{E}_1(\mathbf{E}_2(t))$; equivalently,
$$
\mathbf{E}_{F^{-1}\circ G}(t)\ \overset{d}{=}\ \mathbf{E}_1(\mathbf{E}_2(t)),
$$
where $\mathbf{E}_{F^{-1}\circ G}(t)$ denotes the inverse subordinator with Laplace exponent
$F^{-1}\circ G$. Thus the time change relating the two multi-term equations decomposes into two
consecutive ones: first by $\mathbf{E}_1$ (multi-term to single-term), then by
$\mathbf{E}_2$ (single-term to multi-term).
\end{rem}

\begin{rem}\label{rem:order-gap}
The condition $\gamma_1\leq\alpha_m$ in Theorem \ref{thm:multi-to-multi} is a convenient sufficient
condition which guarantees the Bernstein property of $F^{-1}\circ G$ through the two-step construction,
but it is not necessary: by Proposition \ref{prop:shared-outer}, $F^{-1}\circ G$ can be $\mathcal{BF}$
even when $\gamma_1>\alpha_m$, provided the two characteristics share a common outer symbol.
This happens, for instance, for the telegraph-type equations in Example \ref{ex:fractional-telegraph},
where $F^{-1}\circ G(\lambda)=\lambda^{\beta/\alpha}$ is $\mathcal{BF}$ for all $0<\beta<\alpha$,
including the range $2\beta>\alpha$.
\end{rem}

\begin{example}\label{ex:telegraph-to-fractional-telegraph}
Take $m=n=2$ with $(\alpha_1,\alpha_2)=(2,1)$, $(a_1,a_2)=(1,2a)$, $(\gamma_1,\gamma_2)=(1,1/2)$,
$(b_1,b_2)=(1,2a)$, where $a>0$, and choose $\beta=1$. Then $\gamma_1=1=\alpha_m$, the conditions of
Theorem \ref{thm:multi-to-multi} are satisfied, and
$$
\varphi_1(\lambda)=-a+\sqrt{a^2+\lambda},\qquad \varphi_2(\lambda)=\lambda+2a\lambda^{1/2}.
$$
Note that $\varphi_2$ has the positive drift $1$, but its L{\'e}vy measure is infinite (this is the case
$\gamma_1-\gamma_2=\frac12<1$ of Remark \ref{rem:p2-density}), so the corresponding inverse subordinator
has no atom; an explicit density for such stable subordinators with drift is given in
\cite[Example 3.4]{Meerschaert2008}.
A direct computation gives
$$
\varphi_1(\varphi_2(\lambda))=-a+\sqrt{a^2+\lambda+2a\lambda^{1/2}}=\lambda^{1/2},
$$
so $F^{-1}\circ G=\lambda^{1/2}$ and $p_0(s,t)$ is the density of the inverse $1/2$-stable subordinator:
$$
p_0(s,t)=\frac{1}{\sqrt{\pi t}}\,e^{-\frac{s^2}{4t}},
$$
which coincides with the density of $|\mathbf{B}(t)|$ for a one-dimensional Brownian motion $\mathbf{B}(t)$
with generator $\partial_x^2$.

Now take $A=\partial_x^2$ on $\mathbb{R}$. Then the first equation in \eqref{eq:MTFDE-source} is the
telegraph equation $\partial_t^2 u+2a\partial_t u=\partial_x^2 u$, and the second one is the
time-fractional telegraph equation
$$
\partial_t w+2a D_t^{1/2}w=\partial_x^2 w.
$$
It is known \cite{Orsingher2004,Orsingher2009AOP} that the fundamental solution of the latter equation is
the density of $\mathbf{H}(|\mathbf{B}(t)|)$, where
$\mathbf{H}(t)=\mathbf{V}(0)\int_0^t(-1)^{\mathbf{N}(s)}ds$ is the telegraph process, $\mathbf{N}(t)$ is a
Poisson process with rate $a$, and the two-valued random variable $\mathbf{V}(0)=\pm1$ (each with
probability $1/2$) is independent of $\mathbf{N}(t)$ and $\mathbf{B}(t)$.
Theorem \ref{thm:multi-to-multi} and Remark \ref{rem:multi-multi-stochastic} refine this known identity by
decomposing the time change $|\mathbf{B}(t)|$ into two consecutive inverse subordinators:
$$
\mathbf{H}(|\mathbf{B}(t)|)\ \overset{d}{=}\ \mathbf{H}(\mathbf{E}_1(\mathbf{E}_2(t))).
$$
In flowchart terms, the subordination from the telegraph equation to the time-fractional telegraph equation
factors through the heat equation $\partial_t v=\partial_x^2 v$, and each step carries its own stochastic
time change. See also \cite{Lichenggang2026dual} for a related discussion from the dual-diffusion point of
view.
\end{example}

\section{General nonlocal-in-time equations}\label{sec:general}

\subsection{Equations with kernels of type $(\mathcal{PC})$}\label{subsec:PC}

In this subsection we consider the subordination between two abstract nonlocal-in-time equations with kernels of
type $(\mathcal{PC})$ (Definition \ref{def:PC-kernel}), which include the Caputo fractional derivatives of orders in
$(0,1)$ and their tempered counterparts as special cases.

\begin{thm}\label{thm:PC-subordination}
Let $A$ be the generator of a bounded $C_0$-semigroup.
Consider
\begin{equation}\label{eq:PC-u}
\begin{cases}
D_t(k\ast (u(\cdot)-u_0))(t)=Au(t), \qquad t>0; \\
u(0)=u_0
\end{cases}
\end{equation}
and
\begin{equation}\label{eq:PC-v}
\begin{cases}
D_t(l\ast (v(\cdot)-u_0))(t)=Av(t), \qquad t>0; \\
v(0)=u_0,
\end{cases}
\end{equation}
where both $k(t),l(t)$ are functions of type $(\mathcal{PC})$. Suppose that both problems admit bounded strong solutions and
that there exists a Laplace exponent $\varphi(\cdot)$ such that
\begin{equation}\label{eq:PC-condition}
\varphi(\lambda)\widehat{k}(\varphi(\lambda))=\lambda \widehat{l}(\lambda), \qquad (\lambda>0).
\end{equation}
Then there exists an inverse subordinator $\mathbf{E}_\varphi(t)$ with Laplace exponent $\varphi(\cdot)$ such that
\begin{equation}\label{eq:PC-subordination}
v(t)=\int_0^\infty u(s)P(ds,t), \qquad (t>0),
\end{equation}
where $P(s,t)$ is the distribution of $\mathbf{E}_\varphi(t)$ with
\begin{equation}\label{eq:P-Laplace}
\int_0^\infty e^{-\lambda t}P(ds,t)dt=\frac{\varphi(\lambda)}{\lambda}e^{-s\varphi(\lambda)}ds.
\end{equation}
\end{thm}

\begin{proof}
Since
\[
\begin{cases}
k\ast (u(\cdot)-u_0)(t)\mapsto \widehat{k}(\varphi)   \left(\widehat{u}(\varphi)-\frac{1}{\varphi}u_0 \right) ,   \\
(u(\cdot)-u_0)(0)=0,
\end{cases}
\]
taking the Laplace transform on both sides of \eqref{eq:PC-u} leads to
\[
\varphi\left[ \widehat{k} (\varphi)    \left(\widehat{u}(\varphi)-\frac{1}{\varphi}u_0 \right) \right]-0=A\widehat{u}(\varphi).
\]
It follows that
\begin{equation}\label{eq:u-hat-PC}
\widehat{u}(\varphi)=\widehat{k} (\varphi) \left[\varphi\widehat{k} (\varphi)-A\right]^{-1}u_0.
\end{equation}

Similarly, we have
\begin{equation}\label{eq:v-hat-PC}
\widehat{v}(\lambda)=\widehat{l} (\lambda) \left[\lambda\widehat{l} (\lambda)-A\right]^{-1}u_0.
\end{equation}
(Here we use different letters $\varphi>0$ and $\lambda>0$ for the variables of $\widehat{u}$ and $\widehat{v}$.)

Given the Laplace exponent $\varphi(\cdot)$, there exists an inverse subordinator $\mathbf{E}_\varphi(t)$ with
distribution $P(s,t)$ satisfying \eqref{eq:P-Laplace}.

Next we compute the Laplace transform of $\int_0^\infty u(s)P(ds,t)$:
\begin{align*}
&\int_0^\infty e^{-\lambda t}\left(\int_0^\infty u(s)P(ds,t)\right)\,dt \\
=&\int_0^\infty  u(s)\left( \int_0^\infty e^{-\lambda t}P(ds,t)\,dt\right) \\
=&\int_0^\infty \frac{\varphi(\lambda)}{\lambda}e^{-s\varphi(\lambda)} u(s)\,ds \\
=&\frac{\varphi(\lambda)}{\lambda}\cdot \widehat{k} (\varphi) \left[\varphi\widehat{k} (\varphi)-A\right]^{-1}u_0\\
=&\frac{\lambda\widehat{l} (\lambda)}{\lambda}\cdot \left[\lambda\widehat{l} (\lambda)-A\right]^{-1}u_0\\
=&\widehat{v}(\lambda),
\end{align*}
where we used \eqref{eq:u-hat-PC}, \eqref{eq:PC-condition} and \eqref{eq:v-hat-PC} successively.
The result follows from the uniqueness of the Laplace transform and the right-continuity in $t$ of
both sides, as in the proof of Theorem \ref{thm:MTFDE-to-STFDE}.
\end{proof}

\begin{example}\label{ex:stable-to-stable}
Let $k(t)=g_{1-\alpha}(t)$, $l(t)=g_{1-\beta}(t)$, where $0<\beta<\alpha\leq 1$. Then
\[
\widehat{k}(\lambda)=\lambda^{\alpha-1}, \qquad \widehat{l}(\lambda)=\lambda^{\beta-1}.
\]
Taking $\varphi(\lambda)=\lambda^{\beta/\alpha}$, we have
\[
\varphi(\lambda)\widehat{k}(\varphi(\lambda))=\varphi(\lambda)\cdot \varphi^{\alpha-1}(\lambda)
=\varphi^\alpha(\lambda)=\lambda^{\beta}=\lambda\cdot \lambda^{\beta-1}=\lambda \widehat{l}(\lambda).
\]
Theorem \ref{thm:PC-subordination} then gives
\begin{equation*}
v(t)=\int_0^\infty p(s,t)u(s)\,ds, \qquad (t>0),
\end{equation*}
with $p(s,t)$ the density of the inverse of the $\beta/\alpha$-stable subordinator.
This recovers the classical subordination for fractional diffusion equations \cite{Bajlekova2001,Limiao2010}.
\end{example}

\begin{example}\label{ex:tempered-to-stable}
Let $a>0$, $0<\beta<\alpha\leq 1$, $l(t)=g_{1-\beta}(t)$, and let $k(t)$ be such that
\[
\widehat{k}(\lambda)=\frac{(\lambda+a)^\alpha-a^\alpha}{\lambda}.
\]
Then $k(t)$ is a function of type $(\mathcal{PC})$: indeed, $(\lambda+a)^\alpha-a^\alpha$ is the Laplace exponent of a
tempered $\alpha$-stable subordinator, hence a $\mathcal{BF}$ vanishing at the origin, and
$\widehat{k_1}(\lambda)=\frac{1}{\lambda\widehat{k}(\lambda)}$ defines the corresponding kernel $k_1$.

Let
\[
\varphi(\lambda)=\left(\lambda^\beta+a^\alpha\right)^{1/\alpha}-a.
\]
Then
\[
\varphi(\lambda)\widehat{k}(\varphi(\lambda))=\varphi(\lambda)\cdot \frac{(\varphi(\lambda)+a)^\alpha-a^\alpha}{\varphi(\lambda)}
=\left[\left(\lambda^\beta+a^\alpha\right)^{1/\alpha}-a+a\right]^\alpha-a^\alpha=\lambda^\beta
=\lambda \widehat{l}(\lambda).
\]
Moreover, $\left(\lambda^\beta+a^\alpha\right)^{1/\alpha}$ is $\mathcal{CBF}$ by Proposition \ref{thm:CBF-tempered-power}
with $\gamma=1/\alpha$, since $\beta\cdot\frac1\alpha<1$; subtracting the constant $a$ leaves the derivative
unchanged and yields $\varphi(0+)=0$, so $\varphi$ is $\mathcal{BF}$; moreover,
$\varphi(\lambda)\sim\lambda^{\beta/\alpha}$ as $\lambda\to\infty$, so $\varphi$ is driftless with
$\varphi(\infty)=\infty$.
Applying Theorem \ref{thm:PC-subordination}, we obtain
\begin{equation*}
v(t)=\int_0^\infty p_\varphi(s,t)u(s)\,ds, \qquad (t>0),
\end{equation*}
where $p_\varphi(s,t)$ is the density of the inverse subordinator with Laplace exponent $\varphi(\cdot)$.
In other words, the solution of the $\beta$-order fractional diffusion equation is subordinated to the solution of
the tempered $\alpha$-stable type equation.
\end{example}

\begin{example}[A negative result]\label{ex:stable-to-tempered}
Let $b>0$, $0<\beta<\alpha<1$, $k(t)=g_{1-\alpha}(t)$, and let $l(t)$ be such that
\[
\widehat{l}(\lambda)=\frac{(\lambda+b)^\beta-b^\beta}{\lambda}.
\]
Then $l(t)$ is a function of type $(\mathcal{PC})$.

From the condition \eqref{eq:PC-condition}, we obtain
\[
\varphi(\lambda)=\left[(\lambda+b)^\beta-b^\beta\right]^{1/\alpha}.
\]
However, since $\frac1\alpha>1$, Proposition \ref{prop:tempered-iff} shows that $\varphi(\lambda)$ is {\it not} a
Bernstein function. Therefore, one cannot build a subordination representation of the form
\eqref{eq:PC-subordination} from the $\alpha$-order fractional diffusion equation to the tempered $\beta$-stable
type equation.

In the boundary case $\alpha=1$, however, $k=\delta_0$ (cf. Remark \ref{rem:PC-boundary}), the source equation
reduces to the classical equation $u'=Au$, and
$\varphi(\lambda)=(\lambda+b)^\beta-b^\beta$ is the Laplace exponent of the tempered $\beta$-stable
subordinator; the resulting subordination from the heat equation to the tempered $\beta$-stable type equation
is therefore valid. This is the well-known classical case
\cite{Alrawashdeh2017CMA,KumarVellaisamy2015}.
\end{example}

\begin{example}\label{ex:tempered-to-tempered}
Let $a,b>0$ and $0<\beta<\alpha\leq 1$. Moreover, assume $k(t)$ and $l(t)$ satisfy
\[
\widehat{k}(\lambda)=\frac{(\lambda+a)^\alpha-a^\alpha}{\lambda}, \qquad \widehat{l}(\lambda)=\frac{(\lambda+b)^\beta-b^\beta}{\lambda}.
\]
Then both $k(t)$ and $l(t)$ are functions of type $(\mathcal{PC})$.

Let
\[
\varphi(\lambda)=\left[(\lambda+b)^\beta+a^\alpha-b^\beta\right]^{1/\alpha}-a.
\]
Then we have
\[
\varphi(\lambda)\widehat{k}(\varphi(\lambda))=(\varphi(\lambda)+a)^\alpha-a^\alpha=(\lambda+b)^\beta-b^\beta
=\lambda \widehat{l}(\lambda).
\]
The Bernstein property of $\varphi$ is completely decided by Proposition \ref{prop:shifted-power-iff}:
if $\alpha<1$, then $\varphi$ is $\mathcal{BF}$ if and only if $a^\alpha\geq b^\beta$; if $\alpha=1$, then
$\varphi(\lambda)=(\lambda+b)^\beta-b^\beta$ is $\mathcal{BF}$ for all $a,b>0$. Moreover, $\varphi(0+)=0$,
and $\varphi(\lambda)\sim\lambda^{\beta/\alpha}$ as $\lambda\to\infty$, so $\varphi$ is driftless with
$\varphi(\infty)=\infty$. Hence, whenever $a^\alpha\geq b^\beta$ (and for arbitrary $a,b>0$ if $\alpha=1$),
Theorem \ref{thm:PC-subordination} yields the subordination representation
\[
v(t)=\int_0^\infty p_\varphi(s,t)u(s)\,ds, \qquad (t>0).
\]
\end{example}

\begin{rem}[The remaining case]\label{rem:open-tempered}
In Example \ref{ex:tempered-to-tempered}, the remaining case $a^\alpha<b^\beta$ (with $\alpha<1$) is settled
negatively by Proposition \ref{prop:shifted-power-iff}: $\varphi$ fails to be $\mathcal{BF}$, so no
subordination representation of the form \eqref{eq:PC-subordination} exists between these two tempered
equations in that case. The second-order test at the origin detects only a smaller range: writing
$u(\lambda)=(\lambda+b)^\beta+c$ with $c=a^\alpha-b^\beta$, a direct computation gives
\[
\varphi''(\lambda)=\frac{\beta}{\alpha}\,(\lambda+b)^{\beta-2}\,u(\lambda)^{1/\alpha-2}
\Big[\big(\tfrac{\beta}{\alpha}-1\big)(\lambda+b)^{\beta}+(\beta-1)c\Big],
\]
and the expression in brackets at $\lambda=0$ equals $(1-\beta)(b^\beta-a^\alpha)-(1-\beta/\alpha)\,b^\beta$,
which is positive exactly when
\[
\frac{a^\alpha}{b^\beta}<\frac{\beta(1-\alpha)}{\alpha(1-\beta)},
\]
a threshold strictly smaller than $1$ that tends to $0$ as $\alpha\to1$ (in agreement with the case
$\alpha=1$, where $\varphi$ is always $\mathcal{BF}$). Hence for $a^\alpha<b^\beta$ close to the boundary,
the failure of the Bernstein property escapes the second-order test, and the global analytic
argument in the proof of Proposition \ref{prop:shifted-power-iff} is indispensable. Note also that the
obstruction here concerns powers \emph{larger than one} of the complete Bernstein function
$(\lambda+b)^\beta+c$, in contrast with the operation $f(z^\alpha)^{1/\alpha}$, $\alpha\in(0,1)$,
discussed in Remark \ref{rem:BGT-open-question}.
\end{rem}

\begin{rem}[Tempered stable subordinators and tempered fractional derivatives]\label{rem:tempered-literature}
Examples \ref{ex:tempered-to-stable}--\ref{ex:tempered-to-tempered} involve equations of tempered stable type.
The inverse tempered stable subordinator has been studied in detail in
\cite{Alrawashdeh2017CMA,KumarVellaisamy2015};
in particular, its density exists for every $t>0$, in agreement with Remark \ref{rem:density-allt}.

We point out that the tempered fractional derivative is not defined in a uniform way in the literature; see,
e.g., \cite{LiDengZhao2019,Mali2022MMAS,Ortigueira2021}. The operator $D_t(k\ast(u(\cdot)-u_0))$ used here is
the convolution-type derivative in the sense of \cite{Kochubei2011} associated with the Laplace exponent
$\lambda\widehat{k}(\lambda)=(\lambda+b)^\alpha-b^\alpha$ of the driftless tempered stable subordinator; in
particular it annihilates constants. The shift $-b^\alpha$ is essential: it makes the symbol vanish at the
origin, so that $(\lambda+b)^\alpha-b^\alpha$ is the exponent of a subordinator without killing; the unshifted
symbol $(\lambda+b)^\alpha$ of the Riemann--Liouville-type (substantial) or Caputo-type tempered derivatives
\cite{LiDengZhao2019,Ortigueira2021} corresponds to a killed subordinator and falls outside the framework of
Theorem \ref{thm:PC-subordination}. Even after adding the shift $-b^\alpha u(t)$, the Caputo-type tempered
derivative of \cite{LiDengZhao2019} still has the initial-value term $(\lambda+b)^{\alpha-1}u_0$ in place of
$\frac{(\lambda+b)^\alpha-b^\alpha}{\lambda}u_0$; the two operators coincide only when $b=0$ or $\alpha=1$.
\end{rem}

\subsection{Nonlocal-in-time telegraph equations}\label{subsec:telegraph}

Nonlocal-in-time telegraph equations and the associated telegraph processes with random time have been studied
recently; see \cite{Alegria2023JDE}. In this subsection we establish the subordination between two such equations.

\begin{thm}\label{thm:telegraph-subordination}
Consider
\begin{equation}\label{eq:TE-u}
\begin{cases}
D_t^2(k\ast k\ast (u(\cdot)-u_0))(t)+2a D_t(k\ast (u(\cdot)-u_0))(t)=Au(t), \qquad t>0; \\
u(0)=u_0
\end{cases}
\end{equation}
and
\begin{equation}\label{eq:TE-v}
\begin{cases}
D_t^2(l\ast l\ast (v(\cdot)-u_0))(t)+2a D_t(l\ast (v(\cdot)-u_0))(t)=Av(t), \qquad t>0; \\
v(0)=u_0
\end{cases}
\end{equation}
where both $k(t),l(t)$ are functions of type $(\mathcal{PC})$.
Suppose that both problems admit exponentially bounded strong solutions and that there exists a Laplace exponent
$\varphi(\cdot)$ such that
\[
\varphi(\lambda)\widehat{k}(\varphi(\lambda))=\lambda \widehat{l}(\lambda), \qquad (\lambda>0).
\]
Assume moreover that the solution of \eqref{eq:TE-u} is bounded, i.e., there exists a constant $M$ such that
$\|u(t)\|\leq M$ for all $t\geq0$.

Then there exists an inverse subordinator $\mathbf{E}_\varphi(t)$ with Laplace exponent $\varphi(\cdot)$ such that
\begin{equation}\label{eq:TE-subordination}
v(t)=\int_0^\infty u(s)P(ds,t), \qquad (t>0),
\end{equation}
where $P(s,t)$ is the distribution of $\mathbf{E}_\varphi(t)$.
\end{thm}

\begin{proof}
Since
\[
\begin{cases}
k\ast (u(\cdot)-u_0)(t)\mapsto \widehat{k}(\varphi)   \left(\widehat{u}(\varphi)-\frac{1}{\varphi}u_0 \right) ,   \\
k\ast k\ast (u(\cdot)-u_0)(t)\mapsto \widehat{k}^2 (\varphi)    \left(\widehat{u}(\varphi)-\frac{1}{\varphi}u_0 \right), \\
(u(\cdot)-u_0)(0)=0,
\end{cases}
\]
and the convolutions, together with their first derivatives, vanish at $t=0$,
taking the Laplace transform on both sides of \eqref{eq:TE-u} leads to
\[
\varphi^2\left[ \widehat{k}^2 (\varphi)    \left(\widehat{u}(\varphi)-\frac{1}{\varphi}u_0 \right) \right]
+2a\varphi\left[ \widehat{k} (\varphi)    \left(\widehat{u}(\varphi)-\frac{1}{\varphi}u_0 \right) \right]=A\widehat{u}(\varphi).
\]
It follows that
\begin{equation}\label{eq:u-hat-TE}
\widehat{u}(\varphi)=\frac{\varphi^2\widehat{k}^2 (\varphi)+2a\varphi\widehat{k} (\varphi)}{\varphi}\cdot \left[\varphi^2\widehat{k}^2 (\varphi)+2a\varphi\widehat{k} (\varphi)-A\right]^{-1}u_0.
\end{equation}

Similarly, we have
\begin{equation}\label{eq:v-hat-TE}
\widehat{v}(\lambda)=\frac{\lambda^2\widehat{l}^2 (\lambda)+2a\lambda\widehat{l} (\lambda)}{\lambda}\cdot \left[\lambda^2\widehat{l}^2 (\lambda)+2a\lambda\widehat{l} (\lambda)-A\right]^{-1}u_0.
\end{equation}

Given the Laplace exponent $\varphi(\cdot)$, there exists an inverse subordinator $\mathbf{E}_\varphi(t)$ with
distribution $P(s,t)$ satisfying
\[
\int_0^\infty e^{-\lambda t}P(ds,t)dt=\frac{\varphi(\lambda)}{\lambda}e^{-s\varphi(\lambda)}ds.
\]

Next we compute the Laplace transform of $\int_0^\infty u(s)P(ds,t)$:
\begin{align*}
&\int_0^\infty e^{-\lambda t}\left(\int_0^\infty u(s)P(ds,t)\right)\,dt \\
=&\int_0^\infty  u(s)\left( \int_0^\infty e^{-\lambda t}P(ds,t)\,dt\right) \\
=&\int_0^\infty \frac{\varphi(\lambda)}{\lambda}e^{-s\varphi(\lambda)} u(s)\,ds \\
=&\frac{\varphi(\lambda)}{\lambda}\cdot
   \frac{\varphi^2\widehat{k}^2 (\varphi)+2a\varphi\widehat{k} (\varphi)}{\varphi}\cdot \left[\varphi^2\widehat{k}^2 (\varphi)+2a\varphi\widehat{k} (\varphi)-A\right]^{-1}u_0\\
=&\frac{\varphi^2\widehat{k}^2 (\varphi)+2a\varphi\widehat{k} (\varphi)}{\lambda}\cdot \left[\varphi^2\widehat{k}^2 (\varphi)+2a\varphi\widehat{k} (\varphi)-A\right]^{-1}u_0\\
=&\frac{\lambda^2\widehat{l}^2 (\lambda)+2a\lambda\widehat{l} (\lambda)}{\lambda}\cdot \left[\lambda^2\widehat{l}^2 (\lambda)+2a\lambda\widehat{l} (\lambda)-A\right]^{-1}u_0\\
=&\widehat{v}(\lambda),
\end{align*}
where we used \eqref{eq:u-hat-TE}, the condition $\varphi\widehat{k}(\varphi)=\lambda\widehat{l}(\lambda)$ and
\eqref{eq:v-hat-TE} successively.
The result follows from the uniqueness of the Laplace transform and the right-continuity in $t$ of
both sides, as in the proof of Theorem \ref{thm:MTFDE-to-STFDE}.
\end{proof}

\begin{example}\label{ex:fractional-telegraph}
Let $k(t)=g_{1-\alpha}(t)$, $l(t)=g_{1-\beta}(t)$, where $0<\beta<\alpha\leq 1$. As in Example
\ref{ex:stable-to-stable}, $\varphi(\lambda)=\lambda^{\beta/\alpha}$ is a Laplace exponent satisfying
$\varphi\widehat{k}(\varphi)=\lambda\widehat{l}$. Applying Theorem \ref{thm:telegraph-subordination}, we
obtain $v(t)=\int_0^\infty p(s,t)u(s)\,ds$ ($t>0$), with $p(s,t)$ the density of the inverse of the
$\beta/\alpha$-stable subordinator.

It is worth emphasizing that, in contrast with Theorem \ref{thm:multi-to-multi}, no order-gap condition
between the two equations is required here: the highest order $2\beta$ of the second equation is allowed to
exceed the lowest order $\alpha$ of the first one, which happens exactly when $\beta/\alpha>1/2$.
The reason is that the two equations share the same outer symbol $Q(\mu)=\mu^2+2a\mu$ with the same
damping coefficient $a$: their characteristics are $F(\lambda)=Q(\lambda^\alpha)$ and
$G(\lambda)=Q(\lambda^\beta)$, respectively, so Proposition \ref{prop:shared-outer} (see also Remark
\ref{rem:shared-outer-telegraph}) gives
\[
F^{-1}\circ G(\lambda)=\lambda^{\beta/\alpha},
\]
which is $\mathcal{BF}$ for every $0<\beta<\alpha$.
Equivalently, the condition $\varphi\widehat{k}(\varphi)=\lambda\widehat{l}$ in Theorem
\ref{thm:telegraph-subordination} only matches the first-order symbols $\lambda\widehat{k}(\lambda)$ and
$\lambda\widehat{l}(\lambda)$; the second-order terms then match automatically, since they are the same
quadratic function of the matched symbols.
Thus, within this structured telegraph-type setting, the subordination from
$D_t^{2\alpha}u+2aD_t^{\alpha}u=Au$ to $D_t^{2\beta}v+2aD_t^{\beta}v=Av$ holds for all $0<\beta<\alpha$,
whereas Theorem \ref{thm:multi-to-multi} would require $2\beta\leq\alpha$.
The boundary case $\alpha=1$, $\beta=\frac12$ (i.e., $2\beta=\alpha$) is exactly the case of
Example \ref{ex:telegraph-to-fractional-telegraph}, where the corresponding concrete equations and their
stochastic interpretation via the telegraph process are given, while the range $\beta\in(\frac12,1)$ is beyond
the reach of Theorem \ref{thm:multi-to-multi}.
\end{example}

\begin{rem}\label{rem:APP-comparison}
We close this section with a comparison with the work of Alegr{\'i}a, Poblete and Pozo
\cite{Alegria2023JDE}, where the equation \eqref{eq:TE-u} with $A=c^2\partial_x^2$ and the delta initial
datum was studied.
\begin{itemize}
\item[(i)] Theorem 4.1 of \cite{Alegria2023JDE} corresponds to the special case of Theorem
\ref{thm:telegraph-subordination} in which the source equation is the classical telegraph equation
(i.e., $\widehat{k}\equiv 1$): the condition $\varphi\widehat{k}(\varphi)=\lambda\widehat{l}$ then reads
$$
\varphi(\lambda)=\lambda\widehat{l}(\lambda)=\frac{1}{\widehat{l}_1(\lambda)},
$$
where $(l,l_1)\in(\mathcal{PC})$, and $1/\widehat{l}_1$ is $\mathcal{BF}$ by the $(\mathcal{PC})$ condition
(see \cite[Proposition 4.5]{Pruss1993book}). Moreover, the propagation function used in
\cite{Alegria2023JDE} (with the parameter $\sigma=1$) has the Laplace transform
$\frac{\varphi(\lambda)}{\lambda}e^{-s\varphi(\lambda)}$ with respect to $t$; it therefore coincides with
the distribution $P(s,t)$ of the inverse subordinator $\mathbf{E}_\varphi(t)$ in Theorem
\ref{thm:telegraph-subordination}.
\item[(ii)] Theorem \ref{thm:telegraph-subordination} is more general in three aspects: it requires no
$1$-regularity or sectoriality of the conjugate kernel (needed in \cite{Alegria2023JDE} to obtain a density
of the propagation function and to carry out the Fourier-side arguments), since the representation is
formulated directly with the distribution $P(ds,t)$; it applies to a general operator $A$ and general
initial data, rather than only $A=c^2\partial_x^2$ with the delta datum; and it allows two arbitrary kernels
of type $(\mathcal{PC})$ whenever $\varphi\widehat{k}(\varphi)=\lambda\widehat{l}$ holds for a Laplace
exponent $\varphi$.
\item[(iii)] On the stochastic side, \cite[Theorem 5.1]{Alegria2023JDE} represents the telegraph process
with random time governed by \eqref{eq:TE-u} as $\mathbf{T}(|\mathbf{W}(t)|)$, where $\mathbf{T}(t)$ is the
classical telegraph process and $|\mathbf{W}(t)|$ is built from the folded fundamental solution of a nonlocal
wave-type equation. This process-level representation is complementary to the inverse-subordinator time
change used here; in the time-fractional case both reduce to the time change by $|\mathbf{B}(t)|$ in Example
\ref{ex:telegraph-to-fractional-telegraph} \cite[Example 5.1]{Alegria2023JDE}.
\end{itemize}
\end{rem}

\end{document}